\documentclass[leqno,final]{siamltex}

\usepackage{graphicx}
\usepackage{amsmath,amstext,amssymb}
\usepackage{mathabx}
\newtheorem{remark}{Remark}[section]
\newcommand{\norm}[1]{\left\Vert#1\right\Vert}
\newcommand{\norml}[2]{\Vert#1\Vert_{\bL^2(#2)}}
\newcommand{\normL}[2]{\Vert{\hskip -1.3pt}\vert #1 \vert{\hskip -1.3pt}\Vert_{\bL^2(#2)}}
\newcommand{\norme}[1]{\Vert{\hskip -1.3pt}\vert #1 \vert{\hskip -1.3pt}\Vert_{DG}}

\newcommand{\abs}[1]{\bigl\vert#1\bigr\vert}

\newcommand{\pdaj}[2]{\left\langle \left\{#1\right\}, \left[#2\right]\right\rangle_e}

\newcommand{\set}[1]{\left\{#1\right\}}
\newcommand{\av}[1]{\left\{#1\right\}}
\newcommand{\jm}[1]{\left[#1\right]}

\newcommand{\T}{\mathcal{T}_h}

\newcommand{\hcsta}{\widehat{C}_{\rm sta}}

\newcommand{\db}{\displaybreak[0]}
\newcommand{\nn}{\nonumber}

\newcommand{\al}{\alpha}
\newcommand{\be}{\beta}

\newcommand{\ep}{\varepsilon}
\newcommand{\ga}{\gamma}
\newcommand{\Ga}{\Gamma}
\newcommand{\la}{\lambda}

\newcommand{\na}{\nabla}

\newcommand{\Om}{\Omega}
\newcommand{\pa}{\partial}
\newcommand{\pr}{\prime}

\newcommand{\vp}{\varphi}

\renewcommand{\i}{{\rm\mathbf i}}

\DeclareMathOperator{\re}{{Re}}
\DeclareMathOperator{\im}{{Im}}

\newcommand{\bR}{\mathbf{R}}
\newcommand{\bL}{\mathbf{L}}
\newcommand{\bH}{\mathbf{H}}
\newcommand{\bP}{\mathbf{P}}
\newcommand{\bV}{\mathbf{V}}

\newcommand{\bn}{\mathbf{n}}

\newcommand{\bw}{\mathbf{w}}
\newcommand{\p}{\partial}

\newcommand{\Ome}{\Omega}
\newcommand{\nab}{\nabla}
\newcommand{\Del}{\Delta}

\newcommand{\cT}{\mathcal{T}}

\newcommand{\cE}{\mathcal{E}}
\newcommand{\cF}{\mathcal{F}}
\newcommand{\cJ}{\mathcal{J}}

\newcommand{\Div}{{\rm div\,}}

\newcommand{\Langle}{\bigl\langle}
\newcommand{\Rangle}{\bigr\rangle}

\def\jump#1{[#1]}

\def\avrg#1{\{#1\}}

\newcommand{\cR}{\mathcal R}

\DeclareMathOperator{\ddiv}{{div}}
\DeclareMathOperator{\curl}{\mathbf{curl}}

\newcommand{\bff}{\mathbf{f}}

\newcommand{\bfe}{\mathbf{e}}

\newcommand{\bfg}{\mathbf{g}}

\newcommand{\bfu}{\mathbf{u}}

\newcommand{\bfv}{\mathbf{v}}
\newcommand{\bfw}{\mathbf{w}}
\newcommand{\bfx}{\mathbf{x}}

\newcommand{\bfz}{\mathbf{z}}
\newcommand{\bfE}{\mathbf{E}}

\newcommand{\bfH}{\mathbf{H}}
\newcommand{\bfF}{\mathbf{F}}
\newcommand{\obfF}{\overline{\bfF}}

\newcommand{\bfR}{\mathbf{R}}
\newcommand{\bfW}{\mathbf{W}}
\newcommand{\bfV}{\mathbf{V}}

\newcommand{\bfnu}{\boldsymbol{\nu}}
\newcommand{\bfxi}{\boldsymbol{\xi}}
\newcommand{\bfeta}{\boldsymbol{\eta}}
\newcommand{\bfPhi}{\boldsymbol{\Phi}}
\newcommand{\bfPsi}{\boldsymbol{\Psi}}

\newcommand{\bcV}{\boldsymbol{\mathcal{V}}}

\newcommand{\Eheem}{\mathbf{E}_h^{\mathrm{EEM}}}
\newcommand{\ls}{\lesssim}

\newcommand{\uga}{\underline{\gamma}}

\begin{document}
\title{An absolutely stable discontinuous Galerkin method for the indefinite
time-harmonic Maxwell equations with large wave number}
\markboth{X. FENG AND H. WU}{DG METHODS FOR THE MAXWELL EQUATIONS}

\author{
Xiaobing Feng\thanks{Department of Mathematics, The University of
Tennessee, Knoxville, TN 37996, U.S.A.  ({\tt xfeng@math.utk.edu}).
The work of this author was partially supported by the NSF grants DMS-0710831
and DMS-1016173.} 
\and
Haijun Wu\thanks{Department of Mathematics, Nanjing University, Jiangsu,
210093, P.R. China. ({\tt hjw@nju.edu.cn}). The work of this author was
partially supported by the National Magnetic Confinement Fusion Science Program under grant 2011GB105003 and by the NSF of China grants 10971096, 11071116, 91130004.}}



\maketitle

\begin{abstract}
This paper develops and analyzes an interior penalty discontinuous 
Galerkin (IPDG) method using piecewise linear polynomials for the indefinite
time harmonic Maxwell equations with the impedance boundary condition 
in the three dimensional space. The main novelties of the proposed IPDG method 
include the following: first, the method penalizes not only the jumps of the 
tangential component of the electric field across the element faces
but also the jumps of the tangential component of its vorticity field; second, 
the penalty parameters are taken as complex numbers of negative imaginary parts. 
For the differential problem, we prove that the sesquilinear form 
associated with the Maxwell problem satisfies a generalized 
weak stability (i.e., inf-sup condition) for star-shaped domains.
Such a generalized weak stability readily infers wave-number explicit 
a priori estimates for the solution of the Maxwell problem, which 
plays an important role in the error analysis for the IPDG method. 
For the proposed IPDG method, we show that the discrete sesquilinear form
satisfies a coercivity for all positive mesh size $h$ and wave number $k$ and 
for general domains including non-star-shaped ones. 
In turn, the coercivity easily yields the well-posedness 
and stability estimates (i.e., a priori estimates) for the discrete 
problem without imposing any mesh constraint. Based on these discrete stability 
estimates, by adapting a nonstandard error estimate technique 
of \cite{fw08a}, we derive both the 
energy-norm and the $L^2$-norm error estimates for the IPDG method 
in all mesh parameter regimes including pre-asymptotic regime 
(i.e., $k^2 h\gtrsim 1$). Numerical experiments are also presented 
to gauge the theoretical results and to numerically examine the 
pollution effect (with respect to $k$) in the error bounds.  
\end{abstract}




\section{Introduction}\label{sec-1}

This paper develops and analyzes interior penalty discontinuous Galerkin (IPDG) 
methods for the following time harmonic Maxwell problem:
\begin{alignat}{2} \label{e1.1}
\curl\curl\bfE- k^2\bfE&=\bff &&\qquad\mbox{in }\Omega, \\
\curl\bfE\times\bfnu-\i \lambda \bfE_T&=\bfg 
&&\qquad\mbox{on }\Gamma:=\partial\Ome,\label{e1.2}
\end{alignat}
where $\Ome\subset\bfR^3$ is a bounded domain with
Lipschitz continuous boundary $\partial\Ome$ and of diameter $R$. 
$\bfnu$ denotes the unit outward normal to $\partial\Omega$, $\i:=\sqrt{-1}$, the
imaginary unit, and $\bfE_T=(\bfnu\times\bfE)\times\bfnu$,
the {\em tangential component} of the electric field $\bfE$. $k$, called
{\em wave number}, is a positive constant and $\lambda> 0$
is known as the impedance constant. \eqref{e1.2} is 
the standard impedance boundary condition. Assume that $\bfg\cdot\bfnu=0$,
hence, $\bfg_T=\bfg$.

Problem \eqref{e1.1}--\eqref{e1.2} is a prototypical problem in
electromagnetic scattering (cf. \cite{Colton_Kress99} and the references therein)
and has been used extensively as a model (and benchmark) problem to develop various 
numerical discretization methods including finite element methods 
\cite{Monk03,ZSWX09} and discontinuous Galerkin methods 
\cite{HMP11,HPSS05,HPS04,CLS04,NPC11}, and to develop fast 
solvers (cf. \cite{TW05} and the references therein). 
The above Maxwell problem with large wave number $k$ is numerically 
difficult to solve mainly because of the following two reasons. 
First, the large wave number $k$ implies the small wave length
$\ell:=2\pi/k$, that is, the wave is a short wave and very oscillatory.
It is well known that, in every coordinate direction, one must put some
minimal number of grid points in each wave length in order to resolve the
wave. Using such a fine mesh evidently results in a huge algebraic problem 
to solve regardless what discretization
method is used. Practically, ``the rule of thumb" is to use $6-10$ grid 
points per wave length, which means that the mesh 
size $h$ must satisfy the constraint $hk\ls 1$. To the
best of our knowledge, no numerical method in the literature has been proved to be 
uniquely solvable and to have an error bound under the mesh constraint $hk\ls 1$
for the above Maxwell problem. Moreover, numerical experiments have shown 
that under the mesh condition $hk\ls 1$ the errors of all existing numerical
methods grow as the wave number $k$ increases. This means that
the error is not completely controlled by the product $hk$ and it 
provides strong evidences of the existence of so-called ``pollution"
in the error bounds. It is known now \cite{bs00} that 
the existence of pollution is related to the loss of stability 
of numerical methods with large wave numbers for the scalar wave equation,
which is also expected to be the case for the vector wave equations.
Second, for large wave number $k$, the Maxwell operator is strongly 
indefinite. Such a strong indefiniteness certainly passes onto
any discretization of the Maxwell problem. In other words,  
the stiffness matrix of the discrete problem is not only very 
large but also strongly indefinite. Solving such a large, 
strongly indefinite, and ill-conditioned algebraic problem
is proved to be very challenging and all the well-known iterative 
methods were proved numerically to be either ineffective or divergent
for indefinite wave problems in the case of large wave number 
(cf. \cite{TW05} and the references therein).

This paper is an attempt to address the first difficulty 
mentioned above for the Maxwell equations. In particular, 
our goal is to design and analyze discretization methods which 
have superior stability properties and give optimal rates of convergence   
for the Maxwell problem. Motivated by our previous
experiences with the Helmholtz equation \cite{fw08a,fw08b}, 
we again try to accomplish the goal by developing some interior 
penalty discontinuous Galerkin method for problem
\eqref{e1.1}--\eqref{e1.2}.  The focus of the paper
is to establish the rigorous stability and error analysis 
for the proposed IPDG method, in particular, in 
the preasymptotic regime (i.e., when $k^2h\gtrsim 1$).  
For the ease of presentation and to better present ideas, we confine 
ourselves to only consider the linear element in this paper
and will discuss its high order extensions in a forthcoming paper. 
 
The remainder of this paper is organized as follows. section 
\ref{sec-2} is devoted to the study of the coercivity of the Maxwell 
operator and the wave-number explicit estimates for 
the solution of \eqref{e1.1}--\eqref{e1.2}. We show that
the sesquilinear form associated with the Maxwell problem
satisfies a generalized weak coercivity (i.e., inf-sup
condition). This coercivity in turn readily infers 
the wave-number explicit solution estimates which 
were proved in \cite{Feng10,HMP10}. We note that the proofs 
of both results given in this paper are of independent interest
and refer the reader to \cite{Feng10b} for further 
discussions in the direction. section \ref{sec-3} presents 
the construction of our IPDG method and some simple properties
of the proposed discrete sesquilinear form. section \ref{sec-4}
studies the coercivity of the discrete sesquilinear form and
derives stability estimates for the IPDG solutions. It is proved
that the discrete sesquilinear form satisfies a
coercivity for all mesh size $h>0$ and all wave number $k>0$ 
and for general domains including non-star-shaped ones, 
which is stronger than the generalized 
weak coercivity satisfied by its continuous counterpart. 
All these are possible because of the special design of the 
discrete sesquilinear form and the special property 
$\curl\curl\bfv_h=0$ (element-wise) for all piecewise 
linear functions $\bfv_h$.  
This coercivity in turn readily infers the well-posedness
and stability estimates for the discrete problem without imposing
any mesh constraint. section \ref{sec-5} devotes to the error 
analysis for the proposed IPDG method. By using the discrete 
stability estimates and adapting a nonstandard
error estimate technique of  \cite{fw08a}, 
we derive both the energy-norm and the $L^2$-norm error estimates 
for the IPDG method in all mesh parameter regimes including 
pre-asymptotic regime (i.e., $k^2 h\gtrsim 1$). Finally, 
we present some numerical experiment results in section \ref{sec-6}
to gauge the theoretical results and to numerically 
examine the pollution effect (with respect to $k$) in the error bounds.

\section{Generalized inf-sup condition and stability estimates for PDE solutions} 
\label{sec-2}
The standard space, norm and inner product notation
are adopted in this paper. Their definitions can be found in 
\cite{bs94,ciarlet78}.
In particular, $(\cdot,\cdot)_Q$ and $\langle \cdot,\cdot\rangle_\Sigma$
for $Q\subset \Ome$ and $\Sigma\subset \pa \Ome$ denote the $L^2$-inner product
on {\em complex-valued} $L^2(Q)$ and $L^2(\Sigma)$ spaces, respectively. 
For a given function space $W$, let $\bfW=(W)^3$. In particular,
$\bL^2(\Ome)=(L^2(\Ome))^3$ and $\bH^k(\Ome)=(H^k(\Ome))^3$.
We also define
\begin{align*}
\bH(\curl, \Ome)&:=\bigl\{ \bfv\in \bL^2(\Ome);\, \curl \bfv\in  \bL^2(\Ome) \bigr\},\\
\bH(\ddiv, \Ome)&:=\bigl\{ \bfv\in \bL^2(\Ome);\, \ddiv \bfv\in  L^2(\Ome) \bigr\},\\
\bH(\ddiv_0, \Ome)&:=\bigl\{ \bfv\in \bL^2(\Ome);\, \ddiv \bfv=0 \bigr\},\\
\bcV &:= \bigl\{ \bfv\in \bH(\curl, \Ome);\, \bfv_T \in \bL^2(\Gamma) \bigr\},\\
\hat{\bcV} &:= \bigl\{ \bfv\in \bH(\curl, \Ome);\, \curl \bfv\in \bH(\curl,\Ome),\, \bfv \in \bH(\curl, \Gamma) \bigr\}. 
\end{align*} 

Throughout this paper, the bold face letters are used to denote 
three-dimensional vectors or vector-valued functions, and
$C$ is used to denote a generic positive constant
which is independent of $h$ and $k$. We also use the shorthand
notation $A\lesssim B$ and $B\gtrsim A$ for the
inequality $A\leq C B$ and $B\geq CA$. $A\simeq B$ is a shorthand
notation for the statement $A\lesssim B$ and $B\lesssim A$.

We now recall the definition of star-shaped domains.

\begin{definition}\label{def1}
$Q\subset \bR^3$ is said to be a {\em star-shaped} domain with respect
to $\bfx_Q\in Q$ if there exists a nonnegative constant $c_Q$ such that
\begin{equation}\label{estar}
(\bfx-\bfx_Q)\cdot \bfnu_Q\ge c_Q \qquad \forall \bfx\in\pa Q.
\end{equation}
$Q\subset \bR^3$ is said to be {\em strictly star-shaped} if $c_Q$ is positive.
Where $\bfnu_Q$ denotes the unit outward normal to $\p Q$.
\end{definition}
Throughout this paper, we assume that $\Ome$ is a strictly star-shaped 
domain.  

Introduce the following sesquilinear form on $\bcV\times \bcV$
\begin{equation}\label{e1.3}
a(\bfu,\bfv):= (\curl\bfu,\curl\bfv)_\Ome-k^2(\bfu,\bfv)_\Ome
-\i\lambda\langle \bfu_T,\bfv_T\rangle_\Gamma,
\end{equation}
Then the weak formulation for the Maxwell system \eqref{e1.1}--\eqref{e1.2} 
is defined as seeking $\bfE\in \bcV$ such that
\begin{eqnarray}\label{e1.4}
a(\bfE,\bfv) =(\bff,\bfv)_\Ome+\langle\bfg,\bfv_T\rangle_\Gamma 
\qquad\forall\bfv\in \bcV.
\end{eqnarray}
Using the Fredholm Alternative Principle it can be
shown that problem \eqref{e1.4} has a unique solution
(cf. \cite{Colton_Kress99, Monk03}).

Note that choosing $\bfv=\nabla \psi$ with $\psi\in H^1_0(\Ome)$ shows that $
(k^2\bfE+\bff, \nabla \psi)_\Ome=0,
$
or
\begin{equation}\label{e1.6}
\ddiv (k^2\bfE+\bff)=0 \qquad\text{in }  \Ome.
\end{equation}

Next, we prove that the sesquilinear form $a(\cdot,\cdot)$ satisfies
a generalized weak coercivity which is expressed in terms of a generalized
{\em inf-sup} condition.

\begin{theorem}\label{inf-sup}
Let $\Ome\subset \mathbf{R}^3$ be a bounded star-shaped domain 
with the positive constant $c_\Ome$ and the diameter $R=\mbox{\rm dim}(\Ome)$. 
Then for any $\bfu\in \hat{\bcV}\cap \bfH(\ddiv_0,\Ome)$ there holds the following 
generalized inf-sup condition for the sesquilinear form $a(\cdot,\cdot)$:
\begin{align}\label{eq2.0a}
\sup_{\bfv\in \hat{\bcV}} \frac{ |\im a(\bfu,\bfv)|}{ \|\bfv\|_{E}} 
\,\,+\sup_{\bfv\in \hat{\bcV}} \frac{ |\re a(\bfu,\bfv)|}{ \normL{\bfv}{\Ome}} 
\geq \frac{1}{\gamma} \|\bfu\|_{E},
\end{align}
where 
\begin{align}\label{eq2.0b}
&\gamma :=\max\bigl\{4kR, M \bigr\},\qquad
M:= \frac{4R^2(k^2+\lambda^2)}{\lambda c_\Ome}, \db \\
&\normL{\bfu}{\Ome} :=\Bigl( k^2\|\bfu\|_{\bL^2(\Ome)}^2 
+  k^2 c_\Ome\|\bfu\|_{\bL^2(\Gamma)}^2 \Bigr)^{\frac12}, \db\\
&\|\bfu\|_{E} := \Bigl( k^2 \|\bfu\|_{\bL^2(\Ome)}^2+ k^2 c_\Ome \|\bfu\|_{\bL^2(\Gamma)}^2+\|\curl\bfu\|_{\bL^2(\Ome)}^2 
 +c_\Ome \|\curl\bfu\|_{\bL^2(\Gamma)}^2 \Bigr)^{\frac12}.\label{eq2.0c}
\end{align}

\end{theorem}

\begin{proof}
Let $\bfw:=\bfx-\bfx_\Ome$. Setting $\bfv=\bfu$ in \eqref{e1.3} and taking 
the real and imaginary parts we get
\begin{align}\label{eq2.1}
\re a(\bfu,\bfu) &=\|\curl \bfu\|_{\bL^2(\Ome)}^2 - k^2\|\bfu\|_{\bL^2(\Ome)}^2,\\
\im a(\bfu,\bfu) &=-\lambda \|\bfu_T\|_{\bL^2(\Gamma)}^2. \label{eq2.2}
\end{align}
Alternatively, setting $\bfv=\curl \bfu\times \bfw$ in \eqref{e1.3} (notice that $\bfv\in\bcV$ is a valid test function
for $\bfu\in \hat{\bcV}$),  taking the real part, and using the following integral identity (cf. \cite{Feng10})
\begin{align} \label{eq2.3}
\re (\bfu,\bfv)_\Ome + \frac12 \|\bfu\|_{\bL^2(\Ome)}^2
&+\frac12 \langle \bfw\cdot \bfnu, |\bfu|^2 \rangle_\Gamma \\
&=\re \langle \bfw\times \bfu, \bfu\times \bfnu \rangle_\Gamma
+\re (\ddiv \bfu, \bfu\cdot \bw)_\Ome \nonumber
\end{align}
and the assumption that $\ddiv \bfu=0$, we get
\begin{align}\label{eq2.4}
2\re a(\bfu,\bfv) &=2\re\bigr( \curl\bfu,\curl\bfv)_\Ome 
- 2k^2\re (\bfu,\bfv)_\Ome +2\lambda \im \langle \bfu_T,\bfv_T \rangle_\Gamma \\
&=2\re\bigr( \curl\bfu,\curl\bfv)_\Ome  + k^2 \|\bfu\|_{\bL^2(\Ome)}^2
+k^2 \langle \bfw\cdot \bfnu, |\bfu|^2 \rangle_\Gamma \nonumber\\
&\qquad
-2k^2 \re \langle \bfw\times \bfu, \bfu\times \bfnu \rangle_\Gamma
+2\lambda \im \langle \bfu_T,\bfv_T \rangle_\Gamma. \nonumber
\end{align}

From \eqref{eq2.1} and \eqref{eq2.4} and using the following integral
identity (cf. \cite{Feng10})
\begin{align}\label{eq2.5}
2\re\bigr( \curl\bfu,\curl\bfv)_\Ome 
=\|\curl \bfu\|_{\bL^2(\Ome)}^2 + \langle \bfw\cdot\bfnu, |\curl \bfu|^2 \rangle_\Gamma,
\end{align}
we have
\begin{align}\label{eq2.6}
2k^2 \|\bfu\|_{\bL^2(\Ome)}^2 &= k^2 \|\bfu\|_{\bL^2(\Ome)}^2 + k^2 \|\bfu\|_{\bL^2(\Ome)}^2\\
&=-2\re\bigr( \curl\bfu,\curl\bfv)_\Ome
   -k^2 \langle \bfw\cdot \bfnu, |\bfu|^2 \rangle_\Gamma \nonumber\\
&\qquad
+2k^2 \re \langle \bfw\times \bfu, \bfu\times \bfnu \rangle_\Gamma
-2\lambda \im \langle \bfu_T,\bfv_T \rangle_\Gamma \nonumber  \\
&\qquad
+ 2\re a(\bfu,\bfv) +\|\curl \bfu\|_{\bL^2(\Ome)}^2 - \re a(\bfu,\bfu) \nonumber \db\\
&=-\langle \bfw\cdot\bfnu, |\curl \bfu|^2 \rangle_\Gamma
-k^2 \langle \bfw\cdot \bfnu, |\bfu|^2 \rangle_\Gamma \nonumber\\
&\qquad
+2k^2 \re \langle \bfw\times \bfu, \bfu\times \bfnu \rangle_\Gamma
-2\lambda \im \langle \bfu_T,\bfv_T \rangle_\Gamma \nonumber  \\
&\qquad
+ 2\re a(\bfu,\bfv)-\re a(\bfu,\bfu)  \nonumber  \db\\
&=-\langle \bfw\cdot\bfnu, |\curl \bfu|^2 \rangle_\Gamma
-k^2 \langle \bfw\cdot \bfnu, |\bfu|^2 \rangle_\Gamma \nonumber \\
&\qquad
-2k^2 \langle \bfw\cdot \bfnu, |\bfu\times \bfnu|^2 \rangle_\Gamma
+2k^2 \re \langle \bfw_T\times \bfu, \bfu\times \bfnu \rangle_\Gamma \nonumber \\
&\qquad
-2\lambda\im\langle \bfu_T,\bfv_T \rangle_\Gamma  
+ 2\re a(\bfu,\bfv)-\re a(\bfu,\bfu). \nonumber
\end{align}
Here we have used the decomposition $\bfw=(\bfw\cdot\bfnu) \bfnu +\bfw_T$ 
to obtain the last equality.

On noting that $\langle \bfw_T\times \bfu, \bfu\times \bfnu \rangle_\Gamma 
=\langle \bfu\cdot \bfnu, \bfw_T\cdot \bfu_T \rangle_\Gamma$,
$\|\bfw\|_{L^\infty(\Ome)} \leq R$,  and that $|\bfv|\le|\curl \bfu||\bfw|$, using the star-shaped domain assumption 
and Schwarz inequality we obtain
\begin{align}\label{eq2.7}
2k^2 \|\bfu\|_{\bL^2(\Ome)}^2 
&\le -c_\Ome \|\curl \bfu\|_{\bL^2(\Gamma)}^2
-k^2 c_\Ome \|\bfu\|_{\bL^2(\Gamma)}^2 \\
&\quad
-2 k^2 c_\Ome \|\bfu_T\|_{\bL^2(\Gamma)}^2
+2k^2 R \norml{\bfu}{\Gamma}\norml{\bfu_T}{\Gamma} \nonumber \\
&\quad
+2\lambda R\norml{\bfu_T}{\Gamma}\norml{\curl \bfu}{\Gamma}  
+ 2\re a(\bfu,\bfv)-\re a(\bfu,\bfu). \nonumber \db\\
&\leq -\frac{c_\Ome}2 \|\curl \bfu\|_{\bL^2(\Gamma)}^2
-\frac{k^2 c_\Ome}{2} \|\bfu\|_{\bL^2(\Gamma)}^2
-2 k^2 c_\Ome \|\bfu_T\|_{\bL^2(\Gamma)}^2 \nonumber\\
&\quad
+\frac{2R^2(k^2+\lambda^2)}{c_\Ome} \|\bfu_T\|_{\bL^2(\Gamma)}^2
+ 2\re a(\bfu,\bfv)- \re a(\bfu,\bfu).  \nonumber
\end{align}

Finally, it follows from \eqref{eq2.1}, \eqref{eq2.2} and \eqref{eq2.7} that
\begin{align}\label{eq2.9}
&2k^2 \|\bfu\|_{\bL^2(\Ome)}^2 +2\|\curl \bfu\|_{\bL^2(\Ome)}^2
+c_\Ome \|\curl \bfu\|_{\bL^2(\Gamma)}^2+k^2 c_\Ome \|\bfu\|_{\bL^2(\Gamma)}^2  \\
&
\leq M |\im a(\bfu,\bfu)| + |\re a(\bfu,4\bfv)|, \nonumber
\end{align}
where $\bfv=\curl \bfu\times \bfw$ and $M$ is defined in \eqref{eq2.0b}.

It is easy to check that there holds for $\bfv=\curl \bfu\times \bfw$ 
\[
\normL{\bfv}{\Ome} \leq kR\|\bfu\|_{E}. 
\]
Hence, it follows from \eqref{eq2.9} that
\begin{align}\label{eq2.10}
\frac{ |\im a(\bfu,\bfu)|}{ \|\bfu\|_{E}} 
+\frac{ |\re a(\bfu,4\bfv)|}{ \normL{4\bfv}{\Ome}} 
&\geq \frac{ |\im a(\bfu,\bfu)|}{ \|\bfu\|_{E}}
+ \frac{|\re a(\bfu,4\bfv)|}{ 4kR \|\bfu\|_{E} } \\
& \geq \frac{1}{\gamma} \cdot \frac{ M|\im a(\bfu,\bfu)| + |\re a(\bfu,4\bfv)| }{\|\bfu\|_{E} }
 \geq \frac{1}{\gamma} \|\bfu\|_{E}, \nonumber
\end{align}
where $\gamma= \max\bigl\{4kR, M \bigr\}$ as defined in \eqref{eq2.0b}. 
The proof is complete.
\end{proof}

An immediate consequence of the above generalized {\em inf-sup} 
condition is the following stability estimate for solutions of 
problem \eqref{e1.1}--\eqref{e1.2}.

\begin{theorem}\label{stability}
In addition to the assumptions of Theorem \ref{inf-sup}, assume
that $\bff\in \bfH(\ddiv,\Ome)$ and $\bfg\in \bL^2(\Gamma)$.
Let $\bfE\in \hat{\bcV}\cap \bfH(\ddiv,\Ome)$ be a solution of the variational
problem \eqref{e1.4}.  Then there holds following stability estimate:
\begin{align}\label{e2.1}
\|\curl\bfE\|_{\bL^2(\Ome)} &+k\|\bfE\|_{\bL^2(\Omega)}
+\sqrt{c_\Ome}\|\curl\bfE\|_{\bL^2(\Ga)}+k\sqrt{c_\Ome}\|\bfE\|_{\bL^2(\Gamma)} \\
&\lesssim k^{-1} \gamma M(\bff,\bfg) + k^{-2}\|\ddiv \bff\|_{\bL^2(\Ome)}   \nonumber
\end{align}
for all $k,\lambda> 0$. Where
\begin{align}\label{e2.2}
M(\bff,\bfg) &:= \|\bff\|_{\bL^2(\Ome)} + c_\Ome^{-\frac12}\|\bfg\|_{\bL^2(\Gamma)}.
\end{align}
\end{theorem}

\begin{proof}
Let $\varphi \in H^1_0(\Ome)$ solve 
\begin{equation}\label{e2.3}
\Delta \varphi = -k^{-2} \ddiv \bff \qquad \mbox{in } \Ome.
\end{equation}
Set $\bfF=\nab \varphi$ and $\bfu=\bfE-\bfF$, where $\bfE$ is a solution 
to \eqref{e1.4}. Trivially, we have $\curl \bfF=0$ and $\ddiv \bfF= - k^{-2} \ddiv \bff$
in $\Ome$, and $\bfF_T= \nab_T \varphi =0$ on $\Gamma$.
By \eqref{e1.6} we also have $\ddiv\bfu=\ddiv(\bfE-\bfF)=0$. 
Hence, $\bfu\in\bH(\ddiv_0,\Ome)$. Moreover, since $\bfE$ 
satisfies \eqref{e1.4}, it is easy to verify that $\bfu$ satisfies
\begin{align}\label{e2.4}
a(\bfu,\bfv)=(\bff+k^2\bfF,\bfv)_\Ome + \langle \bfg,\bfv_T\rangle_\Gamma
\qquad \forall \bfv\in \bcV.
\end{align}

Testing \eqref{e2.3} by $\varphi$ and integrating by parts on both sides 
of the resulting equation yield
\[
\|\nab \varphi\|_{\bL^2(\Ome)}^2 = -k^{-2} (\bff, \nab\varphi)_\Ome
\leq k^{-2} \|\bff\|_{\bL^2(\Ome)} \|\nab \varphi\|_{\bL^2(\Ome)}.
\] 
Hence,
\begin{equation}\label{e2.5}
\|\bfF\|_{\bL^2(\Ome)}=\|\nab \varphi\|_{\bL^2(\Ome)}\leq k^{-2} \|\bff\|_{\bL^2(\Ome)}.
\end{equation}
Alternatively, testing \eqref{e2.3} by $\nab\varphi\cdot\bfw =\bfF\cdot \bfw$ 
with $\bfw=\bfx-\bfx_\Ome$, using the following Rellich identity for the Laplacian 
(cf. \cite{Rellich40, Cummings_Feng06}):
\[
2\re (\Del \varphi\, \nab\overline{\varphi}\cdot\bfw) 
= |\nab \varphi|^2 + 2\re \bigl(\ddiv(\nab \varphi\, \nab \overline{\varphi}\cdot \bfw)\bigr)
-\ddiv( \bfw |\nab \varphi|^2 ),
\]
and integrating by parts we get (note that $\bfF_T=0$)
\begin{align*}
-2 k^{-2} (\ddiv \bff, \bfF\cdot \bfw)_\Ome 
&=\|\bfF\|_{\bL^2(\Ome)}^2 
+2\re \langle \bfF\cdot \bfnu, \bfF\cdot\bfw\rangle_\Gamma
- \langle \bfw\cdot \bfnu, |\bfF|^2 \rangle_\Gamma \\
&=\|\bfF\|_{\bL^2(\Ome)}^2 + \langle \bfw\cdot \bfnu, |\bfF|^2 \rangle_\Gamma.  
\end{align*}
Hence, by \eqref{e2.5} and the star-shaped domain assumption we obtain
\begin{align}\label{e2.6}
\|\bfF\|_{\bL^2(\Ome)}^2+c_\Ome \|\bfF\|_{\bL^2(\Gamma)}^2 
&\leq 2 k^{-2}\|\bfw\|_{L^\infty(\Ome)} \|\ddiv \bff\|_{\bL^2(\Ome)} 
\|\bfF\|_{\bL^2(\Ome)} \\
&\leq 2 k^{-4} R\, \|\ddiv \bff\|_{\bL^2(\Ome)}\norml{\bff}{\Ome}.\nn
\end{align}

Finally, by \eqref{e2.5} and Schwarz inequality we get
\begin{align}\label{e2.7}
\bigl| (\bff+k^2\bfF,\bfv)_\Ome + \langle \bfg,\bfv_T\rangle_\Gamma \bigr|
&\leq 2k^{-1} M(\bff,\bfg)\Bigl( k^2 \|\bfv\|_{\bL^2(\Ome)}^2 
+ k^2 c_\Ome \|\bfv_T\|_{\bL^2(\Gamma)}^2 \Bigr)^{\frac12} \\
&\leq 2k^{-1} M(\bff,\bfg)\,\normL{\bfv}{\Ome}. \nonumber
\end{align}
It follows from the generalized inf-sup condition \eqref{eq2.0a},
\eqref{e2.4} and \eqref{e2.7} that
\begin{align}\label{e2.8}
\gamma^{-1} \|\bfu\|_{E} \leq  4k^{-1} M(\bff,\bfg),
\end{align}
which together with \eqref{e2.6} and the relation $\bfu=\bfE-\bfF$ as well
as the definition of the energy norm $\|\bfu\|_{E}$ infer that
(again, note that $\bfF_T=0$)
\begin{align*}
\|\bfE\|_{E} &\leq \|\bfu\|_{E} + \|\bfF\|_{E} \\
&\leq  4k^{-1} \gamma M(\bff,\bfg) + k\big(\|\bfF\|_{\bL^2(\Ome)}^2+c_\Ome \|\bfF\|_{\bL^2(\Gamma)}^2\big)^{\frac12}  \\
&\leq  4k^{-1} \gamma M(\bff,\bfg) 
+ 2R \|\bff\|_{\bL^2(\Ome)} + (2k)^{-2}\,\|\ddiv \bff\|_{\bL^2(\Ome)}.
\end{align*}
Hence, \eqref{e2.1} holds. The proof is complete.
\end{proof}

We conclude this section with a few remarks.

\begin{remark}
Since problem \eqref{e1.1}--\eqref{e1.2} is linear,
the stability estimate \eqref{e2.1} immediately implies
the uniqueness of the problem in the function class
in which the estimate is derived. This provides an
alternative method (to the traditional integral equation
method and the unique continuation method) for establishing
uniqueness (and existence) for the Maxwell problem \eqref{e1.1}--\eqref{e1.2}.
\end{remark}

\begin{remark}
(a) The generalized inf-sup condition \eqref{eq2.0a} is 
a stronger result than a stability estimate for the solution of 
the Maxwell problem. The reason to restrict $\bfu\in \bH(\ddiv_0,\Ome)$ 
in \eqref{eq2.0a} is that $\curl$ operator has a non-trivial kernel.

(b) Stability estimates similar to \eqref{e2.1} were established 
independently early in \cite{Feng10} and \cite{HMP10}. 
\eqref{e2.1} also explicitly shows the dependence on the {\em size} 
and the {\em shape constant} of the domain. 
Such an estimate plays an important role for designing multilevel Schwarz 
preconditioners for discretizations of \eqref{e1.4} and for doing practical 
simulations because in practice the size of the computational domain $\Ome$ is
often taken to be proportional to the wave length.

In addition, not only the sharp wave number-explicit and domain size-explicit 
stability estimate \eqref{e2.1} is obtained as a corollary of the 
generalized inf-sup condition \eqref{eq2.0a}, but also the derivation
reveals some deep insights about the dependence of the solution on 
the datum functions and the domain.

(c) The generalized inf-sup condition \eqref{eq2.0a} provides a guideline 
for constructing ``good" numerical schemes for the Maxwell equations. 
We shall call a discretization method ``a coercivity preserving 
method" if it satisfies a discrete inf-sup condition which mimics
the continuous inf-sup condition. Constructing such a coercivity preserving
IPDG method is one of primary goals of this paper. 

(d) Generalized inf-sup conditions similar to \eqref{eq2.0a} also hold for
the scalar Helmholtz equation and the elastic Helmholtz equations 
(cf. \cite{Feng10b}).
\end{remark}

Based on the above stability estimates in lower norms, one can also derive 
stability estimates in higher norms when the solution $\bfE$ is sufficient 
regular. We state an $H^\delta$-estimate for $\curl\bfE$ below without 
giving a proof (cf. \cite[Remark 4.9]{HMP10}).

\begin{theorem}\label{high_stability}
Suppose that $\Div\bff =0$ and  the solution $\bfE$ of problem \eqref{e1.1}--\eqref{e1.2}
satisfies $\bfE\in \bfH^\delta(\curl,\Ome)$ for $\frac12< \delta\leq 1$. 
Then there holds estimate
\begin{align}\label{e2.100a}
\|\bfE\|_{\bfH^\delta(\curl,\Ome)} \ls (1+\lambda+k) M(\bff,\bfg)
+ \|\bfg\|_{H^{\frac12}(\Ga)},
\end{align}
where
\begin{align}\label{e2.100b}
\bfH^\delta(\curl,\Ome) &:=\bigl\{\bfu\in \bfH^\delta(\Ome);\, 
\curl\bfu\in \bfH^\delta(\Ome) \bigr\}, \\
\|\bfu\|_{\bfH^\delta(\curl,\Ome)} &:=\Bigl( \|\bfu\|_{\bfH^\delta(\Ome)}^2 
+\|\curl \bfu\|_{\bfH^\delta(\Ome)}^2 \Bigr)^{\frac12}.\label{e2.100c}
\end{align}

\end{theorem}

\section{Formulation of discontinuous Galerkin methods}\label{sec-3}
To formulate our IPDG methods, we first need to introduce some notation.
Let $\{\cT_h\}$ be a family of partitions (into tetrahedrons 
and/or parallelepipeds) 
of the domain $\Ome$ parameterized by $h>0$. For any ``element"
$K\in \cT_h$, we define $h_K:=\mbox{diam}(K)$. Similarly, for each 
face $\cF$ of $K\in \cT_h$, define $h_\cF:=\mbox{diam}(\cF)$.
We assume that the elements of $\cT_h$ satisfy the minimal angle 
condition. Let
\begin{eqnarray*}
\cE_h^I&:=& \mbox{ set of all interior faces of $\cT_h$} ,\\
\cE_h^B&:=& \mbox{ set of all boundary faces of $\cT_h$ on $\Ga=\p\Ome$}.
\end{eqnarray*}
We define the jump $\jump{\bfv}$ and average $\avrg{\bfv}$ of $\bfv$ on an interior face
$\cF=\p K\cap \p K^\pr$ as
\[
\jump{\bfv}|_{\cF}:=\left\{\begin{array}{l}
   \bfv|_{K}-\bfv|_{K^\pr}, \mbox{ if the global label of $K$ is bigger},\\
   \bfv|_{K^\pr}-\bfv|_{K}, \mbox{ if the global label of $K^\pr$ is bigger},
\end{array} \right.\;\; \avrg{\bfv}|_{\cF}:=\frac12\bigl( \bfv|_{K}+ \bfv|_{K^\pr} \bigr).
\]
If $\cF\in\cE_h^B$, set $\jump{\bfv}|_{\cF}=\bfv|_{\cF}$ and $\avrg{\bfv}|_{\cF}=\bfv|_{\cF}$. For every
$\cF=\p K\cap \p K^\pr\in\cE_h^I$, let $\bfnu_\cF$ be the unit outward normal
to the face $\cF$ of the element $K$ if the global label of $K$ is bigger
and of the element $K^\pr$ if the other way around. For every 
$\cF\in\cE_h^B$, let $\bfnu_\cF=\bfnu$ the unit outward normal to $\pa\Om$.

To formulate our IPDG methods, we recall the following (local) integration 
by parts formula:
\begin{align}\label{e3.1}
(\curl\bfE,\bfF)_K =(\bfE,\curl\bfF)_K 
- \langle \bfE\times \bfnu_K, \bfF_T\rangle_{\p K}.
\end{align}
where $\bfF_T=(\bfnu_K\times \bfF)\times \bfnu_K$.

Next, multiplying equation \eqref{e1.1} by a 
test function $\obfF$, integrating over $K\in \cT_h$,
using the integration by parts formula \eqref{e3.1}, and 
summing the resulted equation over all $K\in \cT_h$ we get 
\begin{align}\label{e3.2}
\sum_{K\in\cT_h} \bigl( (\curl \bfE,\curl\bfF)_K 
- \langle\curl \bfE\times \bfnu_K, \bfF_T\rangle_{\p K} \bigr)
-k^2 (\bfE,\bfF)_\Ome= (\bff,\bfF)_\Ome.
\end{align} 
To deal with the boundary terms in the big sum, we appeal
to the following algebraic identity. For each interior 
face $\cF=K\cap K'\in \cE_h^I$ there holds
\begin{align}\label{e3.3}
\langle\curl \bfE\times \bfnu_K, \bfF_T\rangle_{\cF}
&+\langle\curl \bfE\times \bfnu_{K'}, \bfF_T\rangle_{\cF} \\
&=\big\langle \jump{\curl \bfE\times\bfnu_{\cF}},\avrg{\bfF_T}\bigr\rangle_{\cF}
+\big\langle \avrg{\curl \bfE\times\bfnu_{\cF}},\jump{\bfF_T}\bigr\rangle_{\cF}.
\nonumber
\end{align}
Substituting identity \eqref{e3.3} into \eqref{e3.2}
after dropping the first term on the right-hand side of \eqref{e3.3}
(because $\jump{\curl \bfE\times \bfnu_{\cF}}|_{\cF}=0$ if $\bfE$ is 
sufficiently regular) yields
\begin{align*}
\sum_{K\in\cT_h} (\curl \bfE,\curl\bfF)_K 
&-\sum_{\cF\in \cE_h^I} \big\langle \avrg{\curl \bfE\times \bfnu_\cF}, 
\jump{\bfF_T}\bigr\rangle_{\cF}\\ 
&-\big\langle \curl \bfE\times \bfnu, 
\bfF_T\bigr\rangle_{\Ga} -k^2 (\bfE,\bfF)_\Ome= (\bff,\bfF)_\Ome,
\end{align*}
 Utilizing the boundary condition \eqref{e1.2} in the third term 
on the left-hand side and adding a ``symmetrization" term then lead 
to the following equation:
\begin{align}\label{e3.4}
&\sum_{K\in\cT_h} (\curl \bfE,\curl\bfF)_K 
-\sum_{\cF\in \cE_h^I} \Bigl( \big\langle \avrg{\curl \bfE\times \bfnu_\cF},
\jump{\bfF_T}\bigr\rangle_{\cF}  \\
+&\epsilon \big\langle \jump{\bfE_T},\avrg{\curl\bfF\times \bfnu_\cF}\bigr\rangle_{\cF}
\Bigr)-\i\lambda \langle \bfE_T, \bfF_T\rangle_{\Ga}
-k^2 (\bfE,\bfF)_\Ome
=(\bff,\bfF)_\Ome + \langle\bfg, \bfF_T\rangle_{\Ga} \nonumber
\end{align}
where $\epsilon=-1, 0, 1$.

The most important and tricky issue for designing an IPDG method is 
how to introduce suitable {\em interior penalty} term(s) 
on the left-hand side of \eqref{e3.4}. Obviously, different 
interior penalty terms will result in different numerical methods. 
As it was proved in \cite{HPSS05}, using the standard interior
penalty terms will lead to IPDG methods which require 
a restrictive mesh constraint to ensure the stability and 
accuracy in the case of large wave number $k$. Inspired 
by our previous work \cite{fw08a} on IPDG methods for the Helmholtz
equation and guided by our stability analysis 
(see section \ref{sec-4}), here we introduce some non-standard 
interior penalty terms into \eqref{e3.4}, which we shall describe 
below, and the IPDG method so constructed will be proved
to be absolutely stable (with respect to wave number $k$ and
mesh size $h$) in the next section. 

To define our IPDG methods, we first introduce
the ``energy" space $\bV$ and the sesquilinear
form $b_h^\epsilon(\cdot,\cdot)$ on $\bV\times \bV$ as follows.
\begin{align*}
\bV&:=\prod_{K\in\cT_h} \bV_K,\quad \bV_K:=\bigl\{ \bfv\in \bH(\curl,K);\, \bfv|_{\p K} \in \bL^2(\partial K),\, 
\curl \bfv|_{\p K} \in \bL^2(\p K) \bigr\}.
\end{align*}
\begin{align}\label{eah}
b_h^\epsilon(\bfu,\bfv):=&\sum_{K\in\cT_h} (\curl\bfu,\curl\bfv)_K 
 \\
&-\sum_{\cF\in\cE_h^I} \Bigl( \big\langle \avrg{\curl \bfu\times \bfnu_\cF},
\jump{\bfv_T}\bigr\rangle_{\cF} +\epsilon\big\langle \jump{\bfu_T},\avrg{\curl\bfv\times \bfnu_\cF}\bigr\rangle_{\cF}
\Bigr)\nn\\
&- \i \bigl(\cJ_0(\bfu,\bfv) 
+\cJ_1(\bfu,\bfv) \bigr), \nn\db\\
\cJ_0(\bfu,\bfv):=&\sum_{\cF\in\cE_h^I}\frac{\ga_{0,\cF}}{h_\cF}\, 
\bigl\langle \jump{\bfu_T},\jump{\bfv_T}\bigr\rangle_\cF,\label{e3.6}\db\\
\cJ_1(\bfu,\bfv):=&\sum_{\cF\in\cE_h^I} \ga_{1,\cF} h_\cF
\bigl\langle\jump{\curl\bfu\times\bn_\cF},\jump{\curl\bfv\times\bn_\cF}
\bigr\rangle_\cF, \label{cJ1}
\end{align}
where $\ga_{0,\cF}$ and $\ga_{1,\cF}$ are nonnegative numbers to be specified later.

\begin{remark}
(a) Clearly, $b_h^\epsilon(\cdot,\cdot)$ is a consistent discretization for
$\curl\curl$ since $(\curl\curl \bfu,\bfv)_\Ome = b_h^\epsilon(\bfu,\bfv)$ 
for all $\bfu\in \bH^2(\Ome)$ and $\bfv\in \bV$ with $\bfv_T|_\Gamma=0$.

(b) The terms in $-\i\bigl(\cJ_0(\bfu,\bfv)+\cJ_1(\bfu,\bfv) \bigr)$ 
are called penalty terms.
The penalty parameters  $-\i\gamma_{0,\cF}$ and $-\i\gamma_{1,\cF}$ are
pure imaginary numbers with {\em negative} imaginary parts. Our analysis still applies if they are taken as complex numbers of negative imaginary parts.

(c) The $\cJ_0$ term penalizes the jumps of the vector field $\bfu$
and the $\cJ_1$ term penalizes the jumps of the 
tangential component of the vector field $\curl \bfu$.
which, to the best of our knowledge, has not been used before in the 
context of IPDG methods for the Maxwell equations.
They play a vital role for our IPDG methods being 
absolutely stable, see section \ref{sec-4}.

(d) $\epsilon=-1,0,1$ correspond to the nonsymmetric, incomplete, and 
symmetric IPDG methods for the Poisson problem. In the remainder of
this paper, we shall only consider the symmetric case $\epsilon=1$ 
and set $b_h(\cdot,\cdot)=b_h^1(\cdot,\cdot)$ for notation brevity.
\end{remark}

With the help of the sesquilinear form $b_h(\cdot,\cdot)$ we now
introduce the following weak formulation for \eqref{e1.1}--\eqref{e1.2}:
Find $\bfE\in \bV\cap \bH(\curl,\Ome)$ such that
\begin{equation} \label{e3.16}
a_h(\bfE,\bfF) =(f,\bfF)_\Ome +\langle \bfg, \bfF_T\rangle_{\Ga} 
\qquad \forall \bfF\in \bV\cap \bH(\curl,\Ome),
\end{equation}
where
\begin{equation} \label{e3.16a}
a_h(\bfE,\bfF):= b_h(\bfE,\bfF) - k^2(\bfE,\bfF)_\Ome 
-\i \lambda \langle \bfE_T,\bfF_T\rangle_{\Ga}.
\end{equation}
From \eqref{e3.4}, it is clear that, if $\bfE\in \bH^2(\Ome)$ is the 
solution of \eqref{e1.1}--\eqref{e1.2}, then \eqref{e3.16} holds 
for all $\bfF\in \bV$.

For any $K\in \cT_h$,  let $P_r(K)$ denote the set of all complex-valued polynomials
whose degrees in all variables (total degrees) do not exceed $r (\geq 1)$.
 We define our IPDG approximation space $\bV_h$ as
\[
\bV_h:=\prod_{K\in \cT_h} \bP_r(K).
\]
Clearly, $\bV_h\subset \bV\subset \bL^2(\Ome)$. But $\bV_h\not\subset 
\bH(\curl,\Ome)$.

We are now ready to define our IPDG methods based on the weak formulation
\eqref{e3.16}: Find $\bfE_h\in \bV_h$ such that for all $\bfF_h\in \bV_h$
\begin{equation}\label{e3.17}
a_h(\bfE_h,\bfF_h)
=(f,\bfF_h)_\Ome + \bigl\langle \bfg, (\bfF_h)_{T} \bigr\rangle_{\Ga}.
\end{equation}

We note that \eqref{e3.17} defines a family of IPDG methods 
for $r\geq 1$. For the ease of presentation and to better 
present ideas, in the rest of this paper we only consider the
case $r=1$, the linear element case.
In the next two sections, we shall study the stability and
error estimates for the above IPDG method with $r=1$. Especially, 
we are interested in knowing how the stability constants and error 
constants depend on the wave number $k$ (and mesh size $h$, of course) 
and what are the ``optimal" relationship between mesh size $h$ and 
the wave number $k$. We remark that the IPDG method with $r=1$ uses piecewise linear polynomials even for Cartesian meshes. By contrast, for the corresponding linear conforming edge element method on Cartesian meshes, the trial functions have to be chosen as piecewise trilinear polynomials. 
We also note that the linear system resulted from \eqref{e3.17} is 
ill-conditioned and strongly indefinite because the coefficient
matrix has many eigenvalues with very large negative real
parts. Solving such a large linear system is another challenging problem
associated with time harmonic Maxwell problems, which will be addressed
in a future work.

For further analysis we introduce the following semi-norms/norms on  $\bV$:
\begin{align} \label{e3.11}
\|\curl\bfv\|_{\bL^2(\cT_h)}^2
:=&\sum_{K\in\cT_h} \norml{\curl \bfv}{K}^2, \db\\
\norm{\bfv}_{DG}^2 :=&\|\curl\bfv\|_{\bL^2(\cT_h)}^2 +\norml{\bfv}{\Om}^2  
\label{e3.12} \\
&+ 
\sum_{\cF\in\cE_h^I} \Bigl( \frac{\ga_{0,\cF}}{h_\cF} \norml{\jm{\bfv_T}}{\cF}^2 
+\ga_{1,\cF} h_\cF \norml{\jm{\curl\bfv\times \bfnu_\cF}}{\cF}^2 \Bigr)
\nonumber\db\\
=& \|\curl\bfv\|_{\bL^2(\cT_h)}^2 +\norml{\bfv}{\Om}^2+\cJ_0(\bfv,\bfv) + \cJ_1(\bfv,\bfv), \nn  \db\\ 
\norme{\bfv}^2 :=&\norm{\bfv}_{DG}^2 +\sum_{\cF\in\cE_h^I}
\frac{h_\cF}{\ga_{0,\cF}}\norml{\av{\curl \bfv\times \bfnu_\cF}}{\cF}^2.
\label{e3.13}
\end{align}

Clearly, the sesquilinear form $b_h(\cdot,\cdot)$ satisfies: 
For any $\bfv\in \bV$
\begin{align} \label{e3.14}
\re b_h(\bfv,\bfv)&=\|\curl\bfv\|_{\bL^2(\cT_h)}^2
-2\re \sum_{\cF\in\cE_h^I} \bigl\langle \avrg{\curl \bfv\times\bfnu_\cF},
\jump{\bfv_T} \bigr\rangle_\cF, \\
\im b_h(\bfv,\bfv)&=-\cJ_0(\bfv,\bfv)-\cJ_1(\bfv,\bfv). \label{e3.15}
\end{align}

\section{Discrete coercivity and stability estimates}\label{sec-4} 
In this section we shall prove that the discrete sesquilinear form
$a_h(\cdot,\cdot)$ satisfies a discrete coercivity, which 
is slightly stronger than the generalized inf-sup 
condition proved in the previous section for the sesquilinear form
$a(\cdot,\cdot)$. Such a discrete coercivity is possible for 
the linear element because $\curl\curl \bfv_h=0$ (defined element-wise) for all
$\bfv_h\in \bfV_h$.  As an immediate corollary of the discrete 
coercivity, we shall derive a priori estimates for solutions of
\eqref{e3.17} for all $h,k>0$, which then infer 
the well-posedness of \eqref{e3.17}. 

We state the first main theorem of this section which 
establishes a coercivity for the discrete sesquilinear 
form $a_h(\cdot,\cdot)$.

\begin{theorem}\label{discrete_coercivity}
Let $\ga_0=\min_{\cF\in\cE_h^I} \set{\ga_{0,\cF}}$, 
$\ga_1=\min_{\cF\in\cE_h^I} \set{\ga_{1,\cF}}$, and
$h_{\min}= \min_{\cF\in\cE_h} \set{h_{\cF}}$. 
Then there exists a constant $0<C<1$ such that
\end{theorem}
\begin{align}\label{e4.5a}
|a_h(\bfu_h,\bfu_h)| &\geq  \frac{C}{\gamma_h} \|\bfu_h\|_{E,h}^2 
\qquad \forall \bfu_h\in \bfV_h
\end{align}
for all $k, \gamma_{0}, \gamma_{1}>0$. Where
\begin{align}\label{e4.5b}
\gamma_h &:= \frac{1}{\lambda h_{\min}} + \frac{1}{\gamma_1 k^2 h_{\min}^2} 
+\frac{1}{\gamma_0}+1, \\
\|\bfu_h\|_{E,h}&:=\Bigl(\|\curl\bfu_h\|_{\bL^2(\cT_h)}^2
+ k^2 \|\bfu_h\|_{\bL^2(\Ome)}^2 \label{e4.5c} \\
&\qquad+\ga_h\big(\cJ_0(\bfu_h,\bfu_h)+\cJ_1(\bfu_h,\bfu_h)+\lambda \|(\bfu_h)_T\|_{\bL^2(\Gamma)}^2\big) \Bigr)^{\frac12}. \nn 
\end{align}

\begin{proof} 
For any $\cF\in \cE_h$, define $\Ome_\cF:=\bigcup\big\{K\in\cT_h; \, \pa K\cap 
\cF\neq \emptyset \big\}$.  By \eqref{e3.16a}, \eqref{e3.14}, and the following trace inequality
\begin{equation}\label{trace_ineq}
\|\{\bfv_h\}\|_{\bL^2(\cF)} \leq C h_\cF^{-\frac12}\|\bfv_h\|_{\bL^2(\Ome_\cF)} 
\qquad\forall \bfv_h\in \bfV_h
\end{equation}
for some $h_\cF$-independent positive constant $C$, we get
\begin{align}\label{e4.6}
&\re a_h(\bfu_h, \bfu_h)
\leq\|\curl\bfu_h\|_{\bL^2(\cT_h)}^2 - k^2\|\bfu_h\|_{\bL^2(\Ome)}^2\\ 
&\hskip 1.2in 
+2\sum_{\cF\in \cE_h^I} 
\|\{\curl\bfu_h\times \bfnu_\cF\}\|_{\bL^2(\cF)} \|[(\bfu_h)_T]\|_{\bL^2(\cF)}\nonumber\\
&\hskip 0.5in
\leq\|\curl\bfu_h\|_{\bL^2(\cT_h)}^2 - k^2\|\bfu_h\|_{\bL^2(\Ome)}^2 \nonumber\\
&\hskip 1.2in 
+C \sum_{\cF\in \cE_h^I} h_\cF^{-\frac12} \|\curl\bfu_h\|_{\bL^2(\Ome_\cF)} 
\|[(\bfu_h)_T]\|_{\bL^2(\cF)}\nonumber\\
&\hskip 0.5in
\leq \frac32 \|\curl\bfu_h\|_{\bL^2(\cT_h)}^2 - k^2\|\bfu_h\|_{\bL^2(\Ome)}^2 
+C \sum_{\cF\in \cE_h^I} h_\cF^{-1} \|[(\bfu_h)_T]\|_{\bL^2(\cF)}^2.
\nonumber 
\end{align}

Since $\bfu_h\in \bfV_h$ is piecewise linear, then $\curl\curl\bfu_h=0$
in each $K\in \cT_h$. By integrating by parts and using the trace inequality
\eqref{trace_ineq} we obtain 
\begin{align*}
&\norml{\curl\bfu_h}{\cT_h}^2
=\sum_{K\in\cT_h} \bigl( \curl\bfu_h, \curl\bfu_h \bigr)_K\\
&\qquad
=\sum_{K\in\cT_h}\Bigl( \bigl(\curl\curl\bfu_h, \bfu_h \bigr)_K 
-\Langle\curl\bfu_h\times\bfnu_K,(\bfu_h)_T\Rangle_{\pa K} \Bigr) \nn  \db\\
&\qquad
=-\sum_{\cF\in\cE_h^B} \Langle\curl\bfu_h\times\bfnu_K,(\bfu_h)_T\Rangle_{\cF} \nn \\
&\qquad\qquad
-\sum_{\cF\in\cE_h^I}\Bigl(\Langle[\curl\bfu_h\times\bfnu_\cF],\{(\bfu_h)_T\}
\Rangle_{\cF}+\Langle \{\curl\bfu_h\times\bfnu_\cF\},[(\bfu_h)_T]\Rangle_{\cF}\Bigr)
\nn  \db\\ 
&\qquad 
\leq C\sum_{\cF\in\cE_h^B} h_\cF^{-\frac12} \|\curl \bfu_h\|_{\bL^2(\Ome_\cF)} 
\|(\bfu_h)_T\|_{\bL^2(\cF)} \nn \\
&\qquad\qquad
+ C\sum_{\cF\in\cE_h^I}  h_\cF^{-\frac12} 
\|[\curl\bfu_h\times\bfnu_\cF]\|_{\bL^2(\cF)} \|\bfu_h\|_{\bL^2(\Ome_\cF)}  \nn \\
&\qquad\qquad
+ C\sum_{\cF\in\cE_h^I}  h_\cF^{-\frac12}
\|\curl\bfu_h\|_{\bL^2(\Ome_\cF)} \|[(\bfu_h)_T]\|_{\bL^2(\cF)} \nn \db\\
&\qquad
\leq \frac13 \|\curl\bfu_h\|_{\bL^2(\cT_h)}^2 + \frac{k^2}{6} \|\bfu_h\|_{\bL^2(\Ome)}^2
+C \sum_{\cF\in\cE_h^B} h_\cF^{-1} \|(\bfu_h)_T\|_{\bL^2(\cF)}^2 \nn\\
&\qquad\qquad
+C\sum_{\cF\in\cE_h^I} h_\cF^{-1}\Bigl( \|[(\bfu_h)_T]\|_{\bL^2(\cF)}^2
+k^{-2} \|[\curl\bfu_h\times\bfnu_\cF]\|_{\bL^2(\cF)}^2  \Bigr). \nn
\end{align*}
Hence,
\begin{align}\label{e4.7}
&2\norml{\curl\bfu_h}{\cT_h}^2 \leq \frac{k^2}{2} \|\bfu_h\|_{\bL^2(\Ome)}^2
+C \sum_{\cF\in\cE_h^B} h_\cF^{-1} \|(\bfu_h)_T\|_{\bL^2(\cF)}^2 \\
&\hskip 1.0in
+C\sum_{\cF\in\cE_h^I} h_\cF^{-1}\Bigl( \|[(\bfu_h)_T]\|_{\bL^2(\cF)}^2
+ k^{-2} \|[\curl\bfu_h\times\bfnu_\cF]\|_{\bL^2(\cF)}^2 \Bigr). \nn
\end{align}
Adding \eqref{e4.6} and \eqref{e4.7} and rearranging the terms yield
\begin{align}\label{e4.8}
&\|\curl\bfu_h\|_{\bL^2(\cT_h)}^2 + k^2\|\bfu_h\|_{\bL^2(\Ome)}^2 \\
&
\leq -2\re a_h(\bfu_h, \bfu_h)
+\frac{C}{\lambda h_{\min}} \lambda \|(\bfu_h)_T\|_{\bL^2(\Gamma)}^2+ 
\frac{C}{\gamma_0}\sum_{\cF\in\cE_h^I}  \frac{\ga_{0,\cF}}{h_\cF} \norml{\jm{(\bfu_h)_T}}{\cF}^2 \nn \\
&\quad+ \frac{C}{\gamma_1 k^2 h_{\min}^2} \sum_{\cF\in\cE_h^I}\ga_{1,\cF} h_\cF \norml{\jm{\curl\bfu_h\times\bfnu_\cF}}{\cF}^2 .\nn 
\end{align}
Therefore,  by the definitions of $\cJ_0(\cdot,\cdot)$ and $\cJ_1(\cdot,\cdot)$ and 
the identity \eqref{e3.15} we get
\begin{align*}
&\|\curl\bfu_h\|_{\bL^2(\cT_h)}^2 + k^2\|\bfu_h\|_{\bL^2(\Ome)}^2 
+\ga_h\big(\cJ_0(\bfu_h,\bfu_h)+\cJ_1(\bfu_h,\bfu_h)+\lambda \|(\bfu_h)_T\|_{\bL^2(\Gamma)}^2\big)\db\\
&\qquad
\leq -2\re a_h(\bfu_h, \bfu_h)
+ C\ga_h\big(\cJ_0(\bfu_h,\bfu_h)+\cJ_1(\bfu_h,\bfu_h)+\lambda \|(\bfu_h)_T\|_{\bL^2(\Gamma)}^2\big) \nn \db\\
&\qquad
= -2\re a_h(\bfu_h, \bfu_h) - C \gamma_h \im a_h(\bfu_h, \bfu_h) \nn  \db\\
&\qquad
\leq C \gamma_h \bigl(|\re a_h(\bfu_h,\bfu_h)|+|\im a_h(\bfu_h, \bfu_h)|\bigr) \nn\db 
\leq  C \gamma_h |a_h(\bfu_h, \bfu_h)|, \nn 
\end{align*}
where $\gamma_h$ is defined by \eqref{e4.5b}. Hence, \eqref{e4.5a} holds.
The proof is completed.
\end{proof}

\begin{remark}
(a) The discrete sesquilinear form $a_h(\cdot,\cdot)$ satisfies a stronger 
coercivity than its continuous counterpart $a(\cdot,\cdot)$ does, 
see  Theorem \ref{inf-sup}. Moreover, the proof of Theorem 
\ref{discrete_coercivity} is simpler than that of Theorem \ref{inf-sup},
all these are possible because of the special form of $a_h(\cdot,\cdot)$ 
and the fact that $\curl\curl \bfv_h=0$ in $K\in \cT_h$ for all piecewise
linear functions $\bfv_h\in \bfV_h$. However, a weak coercivity is only
expected to hold in the case of high order elements.

(b) It is also important to point out that Theorem \ref{discrete_coercivity}
holds without assuming that $\Ome$ is a star-shaped domain.
\end{remark}

An immediate consequence of the above discrete coercivity are the 
following a priori estimates for solutions to the IPDG method \eqref{e3.17}.

\begin{theorem}\label{discrete_stability}
Every solution $\bfE_h$ of the IPDG method \eqref{e3.17} satisfies the 
following stability estimates.
\begin{align*}
&\norml{\curl\bfE_h}{\cT_h}+ k\norml{\bfE_h}{\Om} 
\ls k^{-1}\ga_h  \|\bff\|_{\bL^2(\Ome)}+(\la^{-1}\ga_h)^{\frac12}\|\bfg\|_{\bL^2(\Gamma)}, \\
&\bigl( \cJ_0(\bfE_h,\bfE_h) + \cJ_1(\bfE_h,\bfE_h)+\lambda \|(\bfE_h)_T\|_{\bL^2(\Gamma)}^2  \bigr)^{\frac12}
\ls k^{-1}\ga_h^\frac12  \|\bff\|_{\bL^2(\Ome)}+\la^{-\frac12}\|\bfg\|_{\bL^2(\Gamma)}.
\end{align*}
\end{theorem}

\begin{proof}
By \eqref{e3.17} and Schwarz inequality we get 
\begin{align*}
\abs{a_h(\bfE_h,\bfE_h)}&=\bigl|(\bff,\bfE_h)_\Ome + \langle \bfg,(\bfE_h)_T \rangle_\Gamma \bigr| 
\leq \|\bff\|_{\bL^2(\Ome)} \|\bfE_h\|_{\bL^2(\Ome)} 
  + \|\bfg\|_{\bL^2(\Gamma)} \|(\bfE_h)_T\|_{\bL^2(\Gamma)} \\
& 
\leq \bigl( k^2 \|\bfE_h\|_{\bL^2(\Ome)}^2 
+ \lambda\ga_h \|(\bfE_h)_T^2\|_{\bL^2(\Gamma)}\bigr)^{\frac12}
\bigl( k^{-2} \|\bff\|_{\bL^2(\Ome)}^2 
+ (\lambda\ga_h)^{-1} \|\bfg\|_{\bL^2(\Gamma)}^2 \bigr)^{\frac12} \\
&
\leq \|\bfE_h\|_{E,h}\,  \bigl( k^{-1} \|\bff\|_{\bL^2(\Ome)} 
+ (\lambda\ga_h)^{-\frac12} \|\bfg\|_{\bL^2(\Gamma)} \bigr).
\end{align*}
The desired estimates follow from combining the above inequality
with \eqref{e4.5a}.  The proof is completed.
\end{proof}

The above discrete stability estimates in turn immediately imply the 
well-posedness of the IPDG method \eqref{e3.17}. 

\begin{corollary}\label{existence}
There exists a unique solution to \eqref{e3.17} for any fixed set of 
parameters $k, h_\cF, \ga_{0,\cF}, \ga_{1,\cF}>0$.
\end{corollary}

\section{Error estimates}\label{sec-5}
In what follows, we suppose $\ga_{0,\cF}\simeq \ga_0$ and 
$\ga_{1,\cF}\simeq \ga_1$ for brevity. For simplicity, we assume
that $\Div\bff=0$ and that $\cT_h$ is a quasi-uniform 
partition of $\Ome$ consisting of tetrahedrons. Let $h:=\max\{h_K;\, K\in \cT_h\}$.

\subsection{$\bH(\curl,\Ome)$-elliptic projection and its error estimates}
\label{sec-5.1}
Let $\bfE$ be the solution to problem \eqref{e1.1}--\eqref{e1.2} 
and $\widetilde\bfE_h\in\bV_h$ be its IPDG $\bH(\curl,\Ome)$-elliptic 
projection defined as follows.
\begin{equation}\label{e5.1}
b_h(\bfE-\widetilde\bfE_h,\bfv_h) +(\bfE-\widetilde\bfE_h,\bfv_h)_\Ome=0 
\quad\forall \bfv_h\in \bfV_h.
\end{equation} 

The following lemma establishes the continuity and coercivity 
for the discrete sesquilinear form $b_h(\cdot,\cdot)$.

\begin{lemma}\label{lem5.1} For any $\bfv, \bfw\in\bV$, the sesquilinear form 
$b_h(\cdot,\cdot)$ satisfies
\begin{equation}\label{e5.2}
\abs{b_h(\bfv,\bfw)+(\bfv,\bfw)_\Om},\quad \abs{b_h(\bfw,\bfv)+(\bfw,\bfv)_\Om}\ls\norme{\bfv}\norme{\bfw}.
\end{equation}
In addition, there exists a positive constant $\uga$ such that, for $\ga_0\ge\uga$,
\begin{equation}\label{e5.3}
\re b_h(\bfv_h,\bfv_h)-\im b_h(\bfv_h,\bfv_h)+(\bfv_h,\bfv_h)_\Om
\ge\frac12\norme{\bfv_h}^2 \qquad\forall \bfv_h\in\bV_h.
\end{equation}
\end{lemma}

\begin{proof}
Clearly, \eqref{e5.2} follows from the definitions  \eqref{eah}--\eqref{cJ1},  \eqref{e3.11}--\eqref{e3.13}, and Schwarz inequality. It remains to prove \eqref{e5.3}.

From \eqref{e3.11}--\eqref{e3.15} we have
\begin{align*}
\re b_h(\bfv_h,\bfv_h)&-\im b_h(\bfv_h,\bfv_h)+(\bfv_h,\bfv_h)_\Om\\
=& \norme{\bfv_h}^2
-2\re \sum_{\cF\in\cE_h^I} \pdaj{\curl \bfv_h\times\bfnu_\cF}{(\bfv_h)_T}\\
&-\sum_{\cF\in\cE_h^I}
\frac{h_\cF}{\ga_{0,e}}\norml{\av{\curl \bfv_h\times \bfnu_\cF}}{\cF}^2.
\end{align*}
 It follows from the derivation of \eqref{e4.6} that
there exists a constant $c_0>0$ such that
\begin{align*}
2\re \sum_{\cF\in\cE_h^I}& \pdaj{\curl \bfv_h\times\bfnu_\cF}{(\bfv_h)_T}\\
&\leq \frac{1}{4}\norml{\curl \bfv_h}{\cT_h}^2 +\frac{c_0}{\ga_0}\sum_{\cF\in\cE_h^I}\frac{\ga_{0,\cF}}{h_\cF}
\norml{\jm{(\bfv_h)_T}}{\cF}^2.
\end{align*}
On the other hand, from  \eqref{trace_ineq}, there exists a  constant $c_1>0$ such that, 
\begin{align*}
\sum_{\cF\in\cE_h^I}
\frac{h_\cF}{\ga_{0,e}}\norml{\av{\curl \bfv_h\times \bfnu_\cF}}{\cF}^2\le \frac{c_1}{\ga_0}\norml{\curl \bfv_h}{\cT_h}^2
\end{align*}
Therefore, 
\begin{align*}
\re b_h(\bfv_h,\bfv_h)-\im b_h(\bfv_h,\bfv_h)+(\bfv_h,\bfv_h)_\Om
\ge\Big(1-\frac14-\frac{c_0+c_1}{\ga_0}\Big)\norme{\bfv_h}^2
\end{align*}
which gives \eqref{e5.3} if $\ga_0$ large enough. The proof is completed.
\end{proof}

\begin{remark}\label{rem5.1}
The coercivity and continuity of $b_h(\cdot,\cdot)$ ensure that
the above $\bfH(\curl,\Ome)$-elliptic projection is well defined.
\end{remark}

The following lemma establishes error estimates for $\bfE-\widetilde\bfE_h$.

\begin{lemma}\label{lem5.3}
Suppose problem \eqref{e1.1}--\eqref{e1.2} is $H^2$-regular, then, under the conditions of Lemma~\ref{lem5.1}, there hold
the following estimates:
\begin{align}\label{e5.6}
&\norme{\bfE-\widetilde\bfE_h}
\ls h\,(1+\ga_1)^\frac12 \|\bfE\|_{\bH^1(\curl,\Ome)},\\
&\|\bfE-\widetilde\bfE_h\|_{\bL^2(\Om)} \ls h^2 (1+\ga_1) \cR(\bfE), \label{e5.7}\\
& \|\bfE-\widetilde\bfE_h\|_{\bL^2(\Ga)}\ls h^\frac32 (1+\ga_1) \cR(\bfE), \label{e5.8}
\end{align}
where
\begin{align}\label{e5.7a}
\cR(\bfE) &:= (1+\ga_1)^{\frac12} \|\bfE\|_{\bH^1(\curl,\Ome)}
+ \|\bfE\|_{H^2(\Ome)}.
\end{align}
\end{lemma}

\begin{proof}
{\em Step 1:} It follows from \cite{Monk03, nedelec1986, gm11} that there exists 
$\widehat\bfE_h\in \bV_h\cap \bH(\curl,\Ome)$ (i.e., the conforming N\'ed\'elec interpolation 
of $\bfE$) such that the following estimates hold: 
\begin{align}\label{e5.8a}
\|\bfE-\widehat\bfE_h\|_{\bL^2(\Om)}
&\ls h^2\|\bfE\|_{H^2(\Om)},\\
\|\bfE-\widehat\bfE_h\|_{\bL^2(\Gamma)} &\ls h^{\frac32} \|\bfE\|_{H^2(\Om)},
\label{e5.8b}\\
\|\bfE-\widehat\bfE_h\|_{\bH(\curl,\Om)} &\ls h\|\bfE\|_{\bH^1(\curl,\Ome)},
\label{e5.8c}\\
\norme{\bfE-\widehat\bfE_h}&\ls h\, (1+\ga_1)^\frac12 \|\bfE\|_{\bH^1(\curl,\Ome)},\label{e5.8d}
\end{align}
where \eqref{e5.8d} can be proved by  \eqref{e5.8c}, the commuting property between the curl-conforming interpolation operator and the div-conforming interpolation
operator \cite[Lemma~8.13]{Monk03},  and the trace inequality. 

Let $\bfPhi_h:=\widetilde\bfE_h - \widehat\bfE_h$ 
and $\bfPsi_h := \bfE-\widehat\bfE_h$,
then $\bfE- \widetilde\bfE_h= \bfPsi_h - \bfPhi_h$. By \eqref{e5.1} we have
\begin{align}\label{e5.9}
b_h(\bfPhi_h,\bfPhi_h) +(\bfPhi_h,\bfPhi_h)_\Ome
=b_h(\bfPsi_h,\bfPhi_h) +(\bfPsi_h,\bfPhi_h)_\Ome.
\end{align}

\medskip
{\em Step 2:} From  Lemma \ref{lem5.1} and \eqref{e5.9} we get
\begin{align}\label{e5.10}
\frac12 \norme{\bfPhi_h}^2 
\leq& \re b_h(\bfPhi_h,\bfPhi_h) -\im b_h(\bfPhi_h,\bfPhi_h)+ \norml{\bfPhi_h}{\Om}^2 \\
=&\re \bigl( b_h(\bfPsi_h,\bfPhi_h) + (\bfPsi_h,\bfPhi_h)_\Ome \bigr) - \im \bigl( b_h(\bfPsi_h,\bfPhi_h) + (\bfPsi_h,\bfPhi_h)_\Ome\bigr)   \nn\\ 
\ls& \norme{\bfPhi_h}\, \norme{\bfPsi_h}.\nn  
\end{align}
Therefore, it follows from 
\eqref{e5.8d} that
\begin{align}\label{e5.11}  
\norme{\bfPhi_h}
&\ls  \norme{\bfPsi_h}\ls  h\, (1+\ga_1)^\frac12 \|\bfE\|_{\bH^1(\curl,\Ome)}, 
\end{align}
which together with the relation $\bfE-\widetilde\bfE_h =\bfPsi_h - \bfPhi_h$ and
the triangle inequality immediately infer \eqref{e5.6}.

\medskip
{\em Step 3:} To show \eqref{e5.7}, we first need the following results that can be proved by following  the proof of
\cite[Proposition 4.5]{HPSS05} and their proofs are omitted: for any $\bfv_h\in \bfV_h$ there exists 
$\bfv_h^c\in \bfV_h\cap \bfH(\curl, \Ome)$ such that
\begin{align} 
\|\bfv_h-\bfv_h^c\|_{\bL^2(\Ome)}^2 
&\leq C\sum_{\cF\in\cE_h^I} h_\cF \norml{\jm{(\bfv_h)_T}}{\cF}^2, \label{e5.11c} \\
\|\curl(\bfv_h-\bfv_h^c)\|_{\bL^2(\cT_h)}^2 &\leq C\sum_{\cF\in\cE_h^I} h_\cF^{-1} \norml{\jm{(\bfv_h)_T}}{\cF}^2 . \label{e5.11d}
\end{align}
Let $\bfPhi_h^c\in \bfV_h\cap \bfH(\curl, \Ome)$
be the conforming approximation of $\bfPhi_h$ as defined above. Then it follows from the definition of the norm $\norm{\cdot}_{DG}$ (cf. \eqref{e3.12}), the above two estimates,  and \eqref{e5.11} that
\begin{align} \label{e5.11a}
\|\bfPhi_h-\bfPhi_h^c\|_{\bL^2(\Ome)}& +h\|\curl(\bfPhi_h-\bfPhi_h^c)\|_{\bL^2(\cT_h)}  \ls\ga_0^{-\frac12} h \norm{\bfPhi_h}_{DG}
\ls h^2 \cR(\bfE).
\end{align}
Noting that
\begin{align*}
&\|\bfE-\widetilde{\bfE}_h\|_{\bL^2(\Ome)}^2 
 = \bigl(\bfE-\widetilde{\bfE}_h, \bfE-\widehat{\bfE}_h \bigr)_\Ome
- \bigl(\bfE-\widetilde{\bfE}_h, \bfPhi_h^c \bigr)_\Ome 
- \bigl(\bfE-\widetilde{\bfE}_h, \bfPhi_h-\bfPhi_h^c \bigr)_\Ome,
\end{align*}
we have
\begin{align*}
\|\bfE-\widetilde{\bfE}_h\|_{\bL^2(\Ome)}& 
\leq \|\bfE-\widehat{\bfE}_h\|_{\bL^2(\Ome)} - 
\frac{\bigl(\bfE-\widetilde{\bfE}_h, \bfPhi_h^c \bigr)_\Ome}{\norml{\bfE-\widetilde\bfE_h}{\Ome}} 
+ \|\bfPhi_h-\bfPhi_h^c\|_{\bL^2(\Ome)},
\end{align*}
which together with \eqref{e5.8a} and \eqref{e5.11a} yields
\begin{align}\label{e5.11e}
\|\bfE-\widetilde{\bfE}_h\|_{\bL^2(\Ome)}
\ls  h^2 \cR(\bfE)
-\frac{\bigl(\bfE-\widetilde{\bfE}_h, \bfPhi_h^c \bigr)_\Ome}{\norml{\bfE-\widetilde\bfE_h}{\Ome}}.  
\end{align}

\medskip
{\em Step 4:} We need to bound the last term on the 
right-hand side of \eqref{e5.11e}. Notice that 
$\bfPhi_h^c\in \bfV_h\cap \bfH(\curl,\Ome)$, by 
using a standard duality argument, see Appendix, 
based on the Helmholtz decomposition of $\bfPhi_h^c$, we can show that
\begin{align} \label{e5.11f}
-\frac{\bigl(\bfE-\widetilde{\bfE}_h, \bfPhi_h^c \bigr)_\Ome}{\norml{\bfE-\widetilde\bfE_h}{\Ome}}\ls  (1+\ga_1)h^2\cR(\bfE). 
\end{align}

\medskip
{\em Step 5:} The desired estimate \eqref{e5.7} follows from combing 
\eqref{e5.11e} and \eqref{e5.11f}. Finally, \eqref{e5.8} follows from $\|\bfE-\widetilde\bfE_h\|_{\bL^2(\Ga)}\le\|\bfE-\widehat\bfE_h\|_{\bL^2(\Ga)}+\|\widehat\bfE_h-\widetilde\bfE_h\|_{\bL^2(\Ga)}$, \eqref{e5.8b}, the trace inequality, \eqref{e5.8a}, and \eqref{e5.7}.  The proof is complete.
\end{proof}

\begin{remark}
The $P_1$-conforming N\'ed\'elec edge element (of second type) 
projection $\widehat\bfE_h$ of $\bfE$ 
is introduced and used in the proof to simplify the analysis at the expense
of requiring $\cT_h$ to be a quasi-uniform and conforming mesh. 
We note that the proof is still valid if one replaces the 
$P_1$-conforming N\'ed\'elec edge element 
projection by the $P_1$-IPDG projection without assuming $\cT_h$ 
is a quasi-uniform or conforming mesh. As expected, the new 
proof will be more complicated and technical, and is left for
the interested reader to explore. 
\end{remark}

\subsection{Error estimates for IPDG method \eqref{e3.17}}
The goal of this subsection is to derive error estimates for scheme 
\eqref{e3.17}.  Instead of using the well-known Schatz argument 
\cite{Schatz74,HPSS05,HPS04,ZSWX09},
which is the (only) technique of choice for deriving error estimates 
for indefinite problems in the literature, we shall obtain
our error estimates by exploiting the linearity of the Maxwell equations 
and making strong use of the discrete stability estimates proved in 
Theorem \ref{discrete_stability} 
and the projection error estimates established in Lemma \ref{lem5.3}. 
This new technique, which is adapted from \cite{fw08a}, allows 
us to derive error estimates for 
$\bfe_h:=\bfE-\bfE_h$ without imposing any mesh constraint.

It is easy to check that there holds the following error equation:
\begin{equation*}
a_h(\bfe_h,\bfv_h)=0 
\quad\forall \bfv_h\in \bfV_h.
\end{equation*}
Let $\bfeta_h:= \bfE-\widetilde\bfE_h$ and $\bfxi_h:= \bfE_h-\widetilde\bfE_h$,  
then $\bfe_h=\bfeta_h-\bfxi_h$. From \eqref{e5.1} we get
\begin{align}\label{e5.13}
a_h(\bfxi_h,\bfv_h) &=a_h(\bfeta_h,\bfv_h) =b_h(\bfeta_h,\bfv_h) -k^2 (\bfeta_h,\bfv_h)_\Ome
-\i\lambda \Langle (\bfeta_h)_T,(\bfv_h)_T \Rangle_\Ga  \\
&=-(k^2+1) (\bfeta_h,\bfv_h)_\Ome-\i\lambda \Langle (\bfeta_h)_T,(\bfv_h)_T \Rangle_\Ga \qquad\forall \bfv_h\in \bfV_h, \nn
\end{align}
The above equation implies that $\bfxi_h\in \bfV_h$ is the solution 
of scheme \eqref{e3.17} with the source functions
$\bff = -(k^2+1) \bfeta_h$ and $\bfg=-\la(\bfeta_h)_T$. Hence, an application of
Theorem~\ref{discrete_stability} and Lemma~\ref{lem5.3} immediately
infers the following estimate for $\bfxi_h$.
 
\begin{lemma}\label{lem5.4}
Under the conditions of Lemma~\ref{lem5.1}, there holds
\begin{align}\label{e5.14} 
&\|\bfxi_h\|_{DG} + k \|\bfxi_h\|_{\bL^2(\Ome)} 
 \ls \hcsta (1+\ga_1) \big(k^2 h^2+\la\, h^{\frac32}\big) \cR(\bfE),
\end{align}
where
\begin{align}\label{e5.15}
\hcsta:= \max\Big(k^{-1}(1+\ga_h), \big(\la^{-1}(1+\ga_h)\big)^{\frac12}\Big), 
\end{align}
and $\ga_h$ is defined by \eqref{e4.5b}.

\end{lemma}

By Lemmas \ref{lem5.3} and \ref{lem5.4} and the triangle inequality 
we then obtain the following main theorem of this section.

\begin{theorem}\label{main_theorem}
Let $\bfE$ and $\bfE_h$ be the solutions to problem \eqref{e1.1}--\eqref{e1.2}
and scheme \eqref{e3.17}, respectively. Assume $\bfE\in \bH^2(\Ome)$. 
Then, under the conditions of Lemma~\ref{lem5.1},  there hold the following error estimates:
\begin{align}\label{e5.16}
&\|\bfE-\bfE_h\|_{DG}\ls  \big(h +\hcsta (1+\ga_1) \big(k^2 h^2+\la\, h^{\frac32}\big)\big) \, \cR(\bfE),\\
&\|\bfE-\bfE_h\|_{\bL^2(\Ome)}
\ls \bigl(h^2+\hcsta k^{-1} \big(k^2 h^2+\la\, h^{\frac32}\big)\bigr) (1+\ga_1)
\, \cR(\bfE). \label{e5.17}
\end{align}
\end{theorem}

To bound $\cR(\bfE)$ in terms of the source functions $\bff$ and $\bfg$,
we need to bound $\|\bfE\|_{\bfH^2(\Ome)}$ and $\|\bfE\|_{\bfH^1(\curl,\Ome)}$ 
by the source functions.  To the end, we appeal to the solution estimate 
\eqref{e2.100a} to get
\begin{align}\label{e5.18}
\cR(\bfE) \ls  (\lambda+k)M(\bff,\bfg) + \|\bfg\|_{H^{\frac12}(\Gamma)} 
\end{align}

Substituting \eqref{e5.18} into \eqref{e5.16} and \eqref{e5.17} yields 
the following explicit in all parameter error bounds for $\bfE-\bfE_h$.

\begin{corollary}\label{cor_main} Suppose $k, \la\gtrsim 1$, and $ 0<\ga_1\ls 1$.
Under the assumptions of Theorem \ref{main_theorem}, there exist constants $C_1$ and $C_2$ independent of $k, \la$, and $h$ such that
\begin{align}\label{e5.19}
\|\bfE-\bfE_h\|_{DG}&\le   C_1(k+\la) h + C_2 \hcsta (k+\la)\big(k^2 h^2+\la\, h^{\frac32}\big),  \\
\|\bfE-\bfE_h\|_{\bL^2(\Ome)}
&\le C_1(k+\la) h^2+ C_2 \hcsta k^{-1}(k+\la)\big(k^2 h^2+\la\, h^{\frac32}\big).\label{e5.20}
\end{align}
\end{corollary}

\begin{remark} \label{r5.3}
(a) If $\la=O(k)$ and $h$ is in the pre-asymptotic range given by $k^2h\gtrsim 1$, then $\la\, h^{\frac32}\ls k\, h^{\frac32}\ls k^2h^2$ and the $H^1$-estimate \eqref{e5.19} becomes
\[\|\bfE-\bfE_h\|_{DG}\le   C_1k h + C_2 \hcsta k^3 h^2.\]

(b) For asymptotic error estimates we refer to \cite[section 7.2]{Monk03}. When $k^3h^2$ is small, it is possible to improve the discrete stability estimates as well as the error estimates via the technique of stability-error iterative improvement from \cite{fw08b, Wu11}. 
\end{remark}

\section{Numerical experiments}\label{sec-6}
Throughout this section, we consider the following Maxwell problem on 
the unit cube $\Om=(0,1)\times(0,1)\times(0,1)$:
\begin{align}
\curl\curl\bfE- k^2\bfE&=\mathbf{0} \qquad\mbox{in }\Omega,\label{e7.1} \\
\curl\bfE\times\bfnu-\i k \bfE_T&=\bfg 
\qquad\mbox{on }\Gamma:=\partial\Ome.\label{e7.2}
\end{align}
where $\bfg$ is so chosen that the exact solution is
$\bfE=\big(e^{\i k z}, e^{\i k x}, e^{\i k y}\big)^T.
$
Notice that we have chosen $\la=k$ for simplicity. 

For any positive integer $m$, let $\cT_{1/m}$ denote the Cartesian mesh that 
consists of $m^3$ congruent cubes of edge length $h=1/m$. We adopt the IPDG method 
using piecewise linear polynomials. We remark that the 
number of total DOFs of the IPDG method on $\cT_{1/m}$ is 
$12m^3$ which is the about twice of that of the corresponding conforming
edge element method (EEM) which uses piecewise trilinear polynomials.

\subsection{Stability}\label{ssec-1} Given a Cartesian mesh $\cT_h$, recall 
that $\bfE_h$ denotes the IPDG solution. Let $\Eheem$ denotes the trilinear conforming
edge element approximation of the problem \eqref{e7.1}--\eqref{e7.2}.
In this subsection, we use the following penalty parameters in
the IPDG method \eqref{e3.17}:
\begin{equation}\label{e7.4}
\ga_{0,\cF}\equiv\ga_0=100\quad\text{ and }\quad \ga_{1,\cF}\equiv\ga_1=0.1 \quad\forall \cF\in\cE_h^I.
\end{equation}

We plot in Figure~\ref{fsta} the following two ratios  
\[
\dfrac{\norm{\bfE_h}_{H(\curl,\cT_h)}}{\norm{\bfE}_{H(\curl,\cT_h)}} \qquad
\mbox{and}\qquad
\dfrac{\norm{\Eheem}_{H(\curl,\cT_h)}}{\norm{\bfE}_{H(\curl,\cT_h)}} 
\]
versus $k$ for $k=1, 2, \cdots, 200$ with $h=0.1, 0.05$, respectively.  It is 
shown that 
\[
\norm{\bfE_h}_{H(\curl,\cT_h)}\ls\norm{\bfE}_{H(\curl,\cT_h)},
\]
which is also implied by Theorem~\ref{discrete_stability} and 
Theorem~\ref{stability}. The $H(\curl)$-norm of the edge element 
solution oscillates for $k$ near $3/h$ but is still bounded 
by $\norm{\bfE}_{H(\curl,\cT_h)}$. 
\begin{figure}[hbt]
\centerline{\includegraphics[scale=0.50]{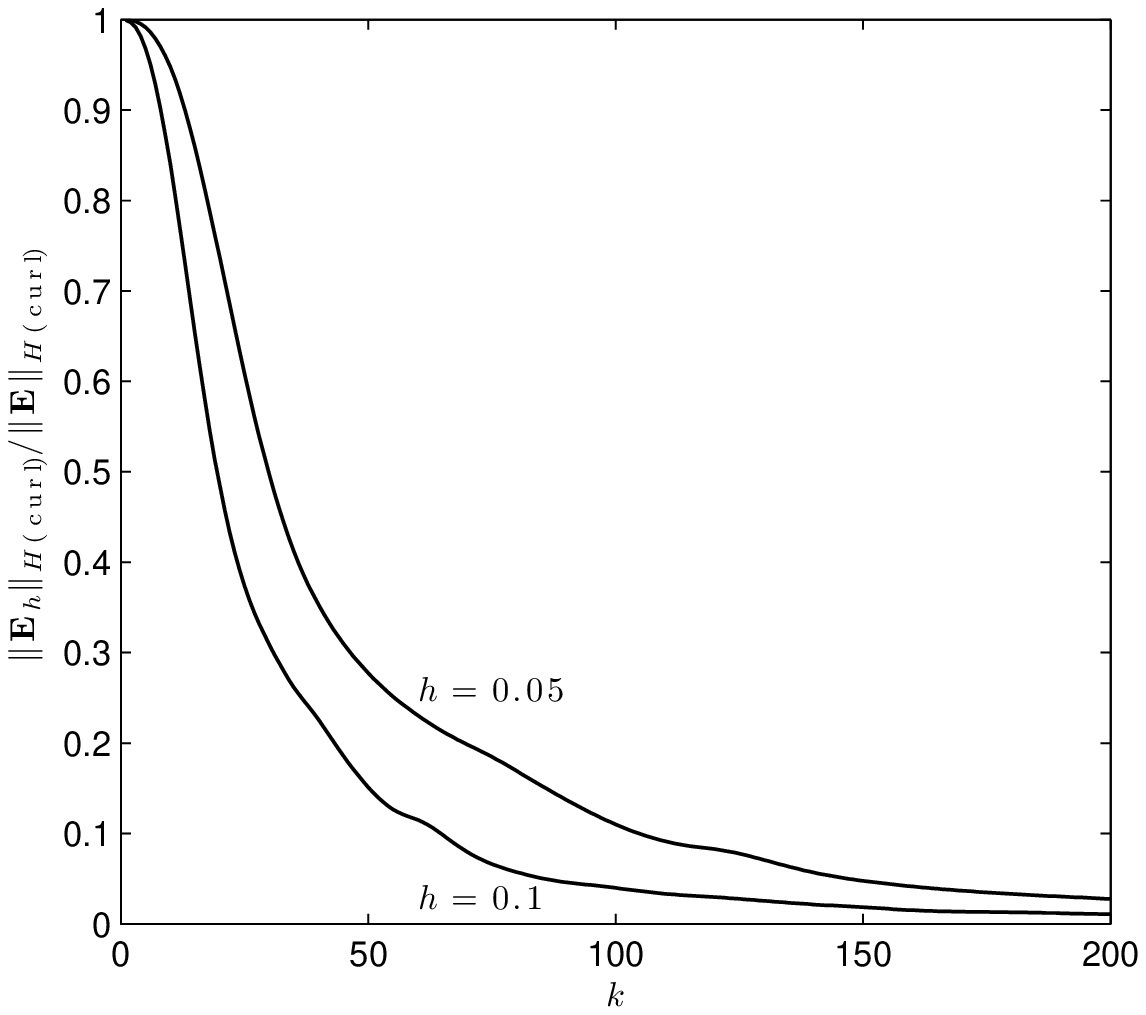} 
\includegraphics[scale=0.50]{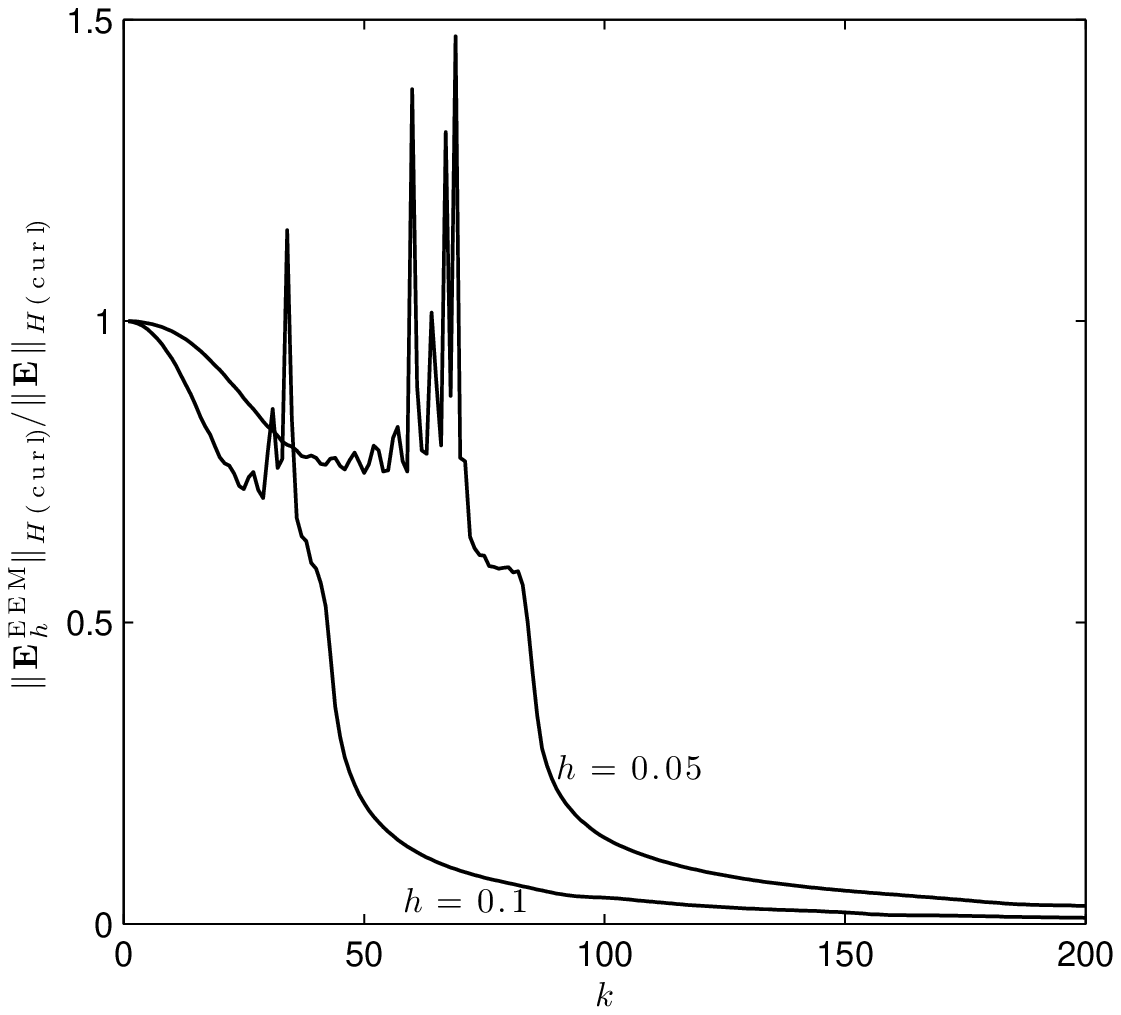}}
\caption{$\norm{\bfE_h}_{H(\curl,\cT_h)}\big/\norm{\bfE}_{H(\curl,\cT_h)}$ (left) 
and $\norm{\Eheem}_{H(\curl,\cT_h)}\big/\norm{\bfE}_{H(\curl,\cT_h)}$ (right) 
versus $k$ for $k=1, 2, \cdots, 200$ with $h=0.1, 0.05$, respectively.}\label{fsta}
\end{figure}

\subsection{Error estimates}\label{ssec-2} In this subsection, we use 
the same penalty parameters as given in \eqref{e7.4}.
In the left graph of Figure~\ref{ferr1}, the relative $H(\curl)$-error 
of the IPDG solution and the relative $H(\curl)$-error of the edge element
interpolant are displayed in one plot. When the mesh size is decreasing, 
the relative error of the
IPDG solution stays around $100\%$ before it is less than $100\%$,
then decays slowly on a range increasing with $k$, and then decays at a
rate greater than $-1$ in the log-log scale but converges as fast as the
edge element interpolant (with slope $-1$) for small $h$. The relative
error grows with $k$ along line $k h=1.$ By contrast, as shown in 
the right of Figure~\ref{ferr1}, the relative error of the finite 
element solution first stay around $100\%$ but oscillates for large $k$, 
then decays at a rate greater than $-1$ in the log-log scale but converges 
as fast as the edge element interpolant (with slope $-1$) for small $h$. 
The relative error of the edge element solution also grows with $k$ 
along line $k h=1$. 
\begin{figure}[ht]
\centerline{
\includegraphics[scale=0.45]{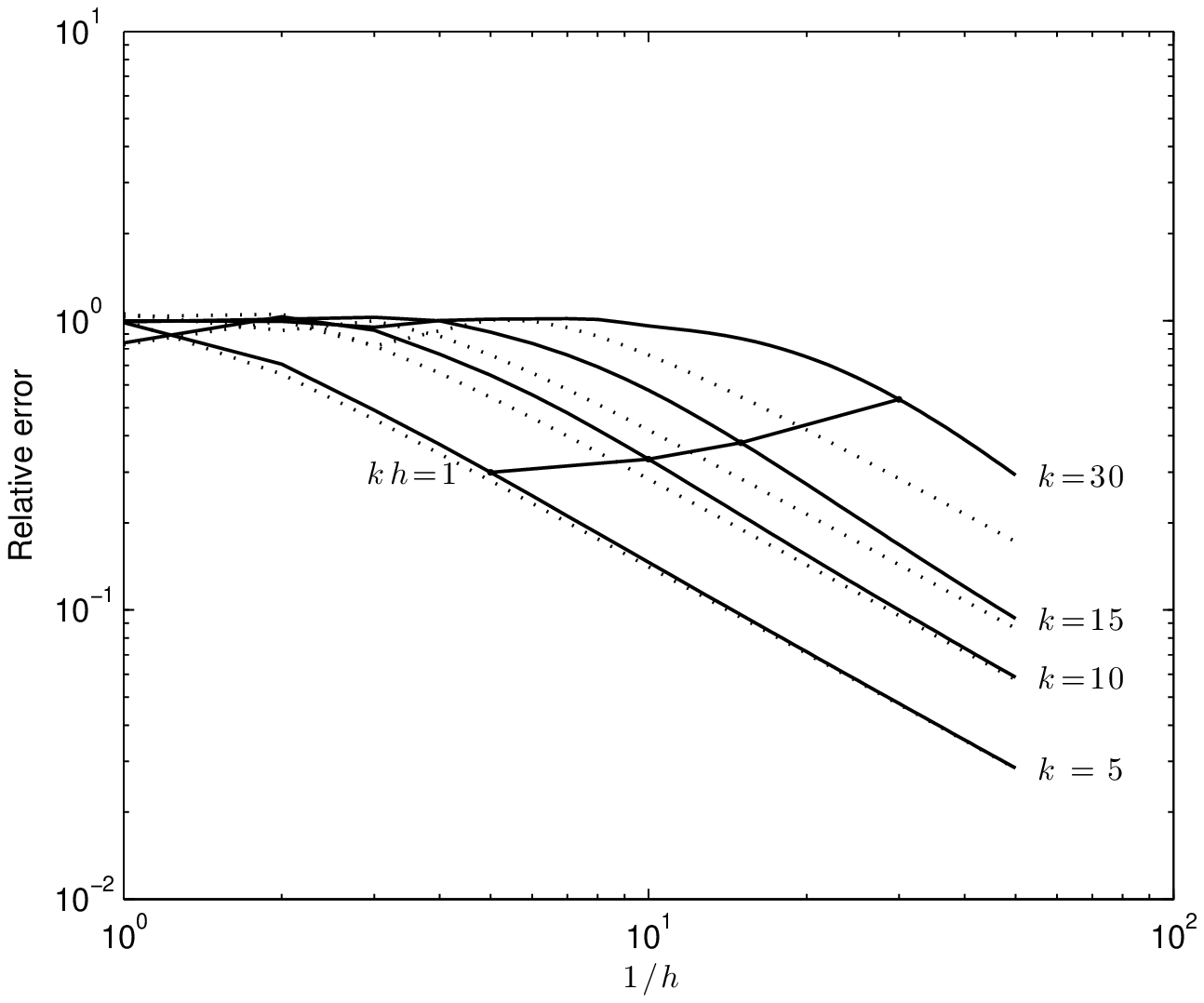}
\includegraphics[scale=0.45]{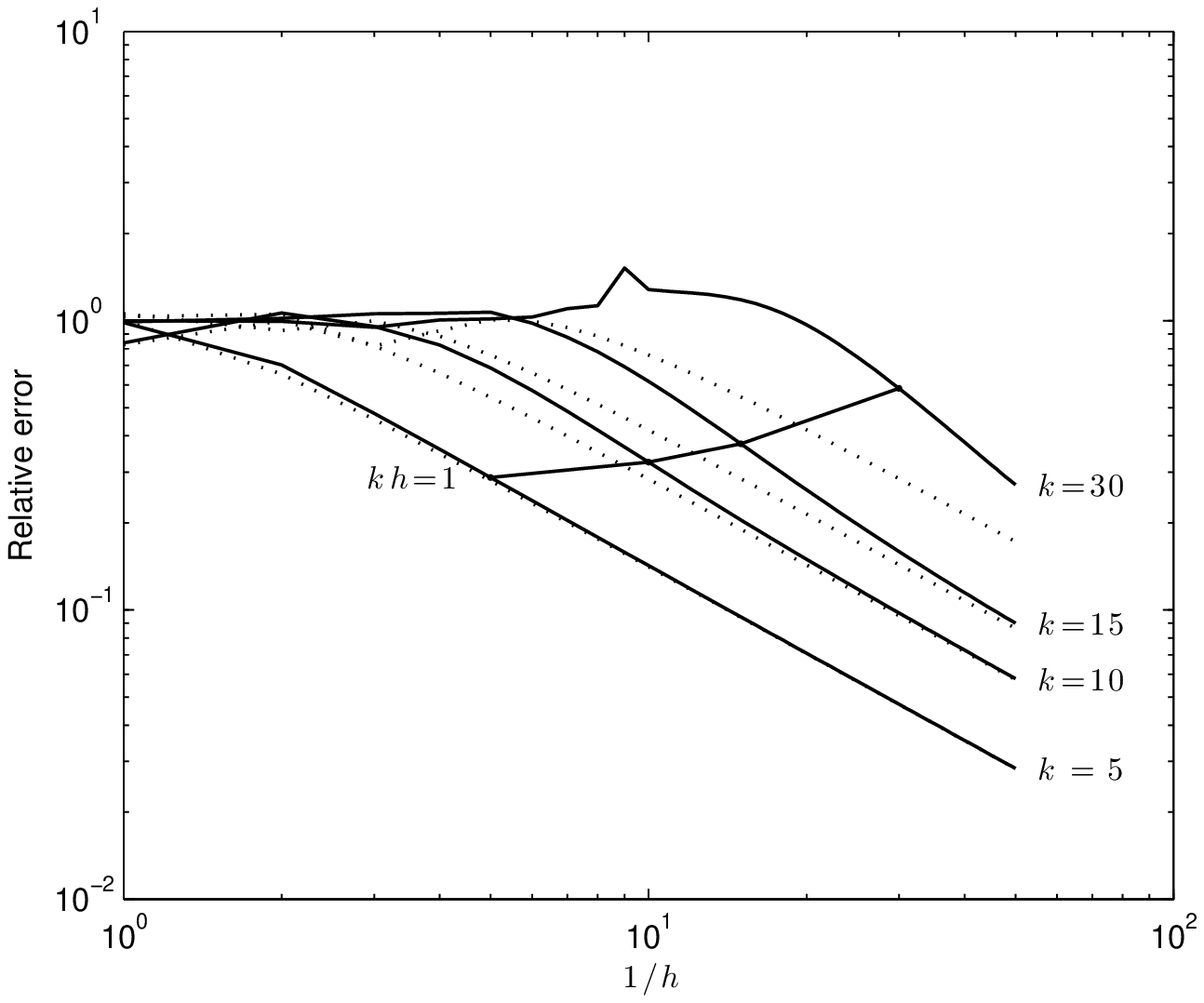}
}
\caption{Left graph: the relative error of the IPDG solution  with parameters
given in \eqref{e7.4} (solid) and the relative error of the edge element
interpolant (dotted) in $H(\curl)$-norm for $k=5, k=10, k=15,$ and $k=30$,
respectively. The dashed line gives reference slope of $-1$. Right graph: 
corresponding plots for edge element solutions.}\label{ferr1}
\end{figure}

Unlike the error of the edge element interpolant, both the error of the 
IPDG solution and that of the edge element solution are not controlled by
the magnitude of $k h$ as indicated by the two graphs in Figure~\ref{ferr2}. 
It is shown that when $h$ is determined according to the ``rule of thumb", 
the relative error of the IPDG solution keeps less than $100\%$  which means 
that the IPDG solution has some accuracy even for large $k$, while the 
edge element solution is unusable for large $k$.   We remark that the 
accuracy of the IPDG solution can be further improved by tuning  
the penalty parameter $\i\ga_1$, see Subsection~\ref{ssec-3} below.

\begin{figure}[ht]
\centerline{
\includegraphics[scale=0.48]{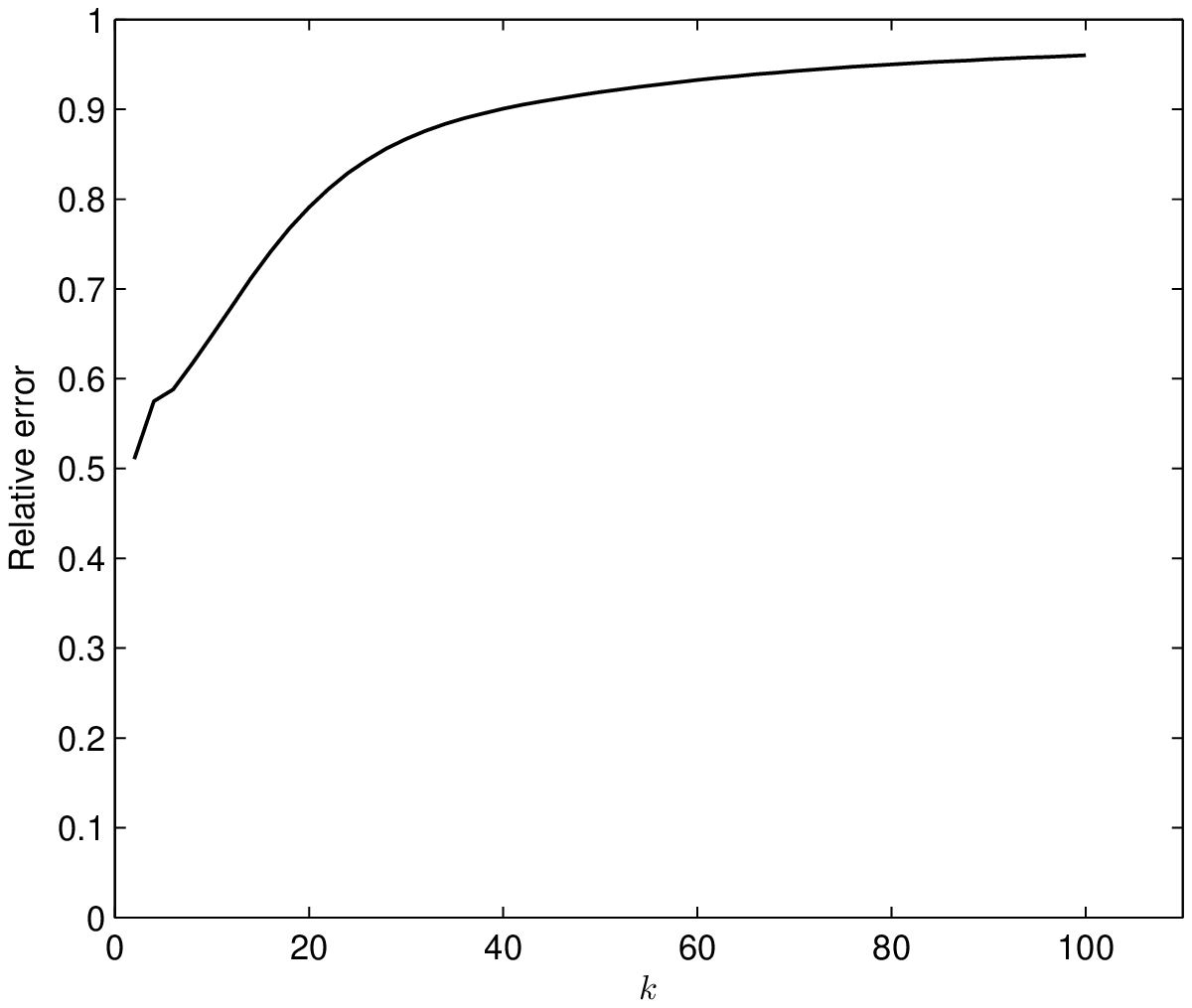}
\includegraphics[scale=0.48]{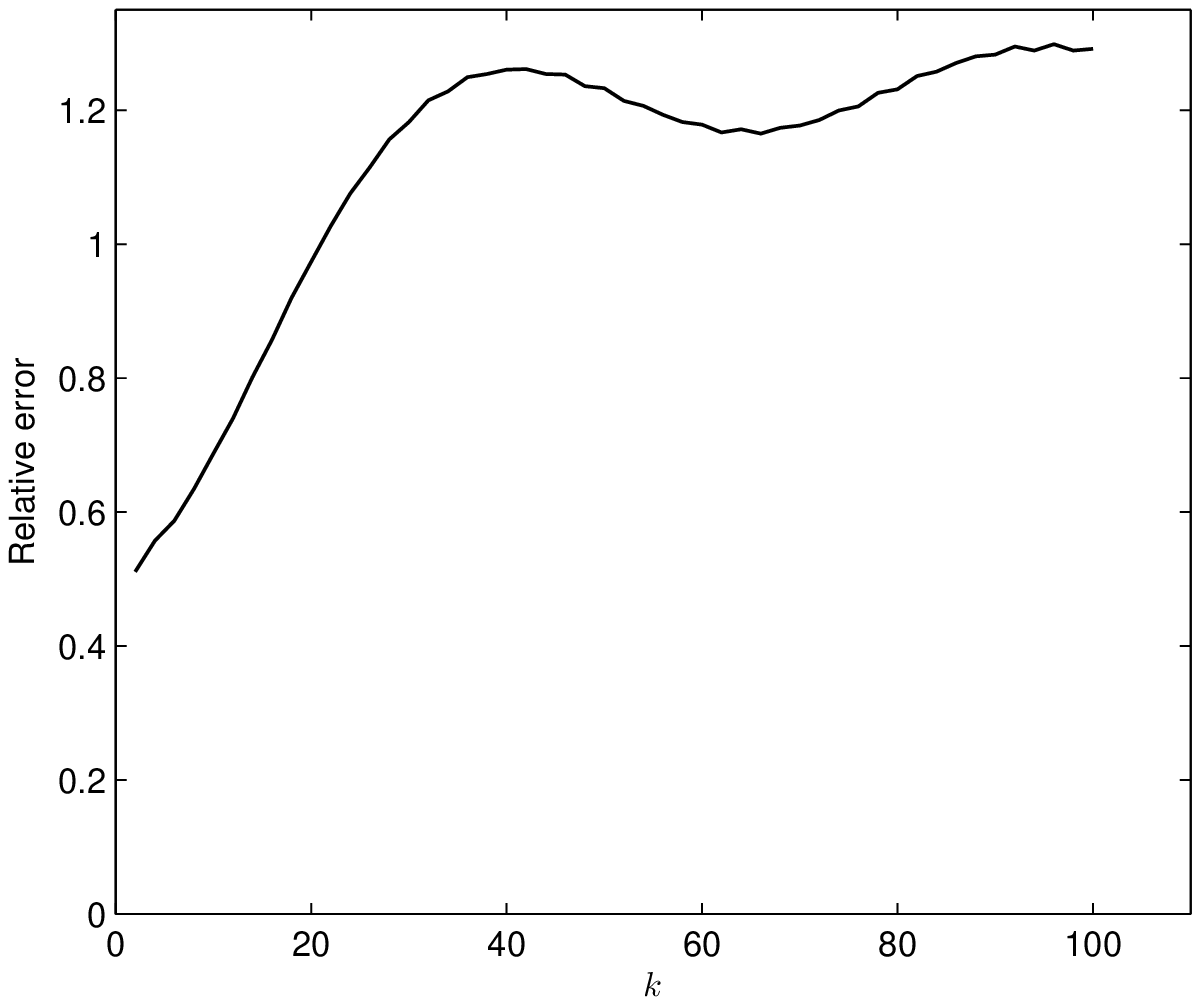}
}
\caption{The relative error of the IPDG solution (left) with parameters given
in \eqref{e7.4} and that of the edge element solution (right) in $H(\curl)$-norm 
computed for $k=2, 4, \cdots, 72$ with mesh size $h$  determined by $kh=2$.}
\label{ferr2}
\end{figure}


Next we verify more precisely the pollution errors. To do so, we recall 
the definition of the critical mesh size with respect to a given relative 
tolerance (cf. \cite[Definition 7.1]{Wu11}). 
\begin{definition}
Given a relative tolerance $\ep$ and a wave number $k$, the critical mesh 
size $h(k,\ep)$ with respect to the relative tolerance $\ep$  is defined 
by the maximum mesh size such that the relative error of the IPDG solution 
(or the edge element solution) in $H(\curl)$-norm is less than or equal to $\ep$.
\end{definition}

It is clear that, if the pollution terms are of order $k^\be h^\al$, 
then $h(k,\ep)$ should be proportional to $k^{-\be/\al}$ for $k$ large 
enough. Figure~\ref{ferr3} which plots $h(k,0.5)$ versus $k$ for the 
IPDG solution (left) with parameters given in \eqref{e7.4} and for the edge 
element solution (right), respectively. They all decay at a 
rate of $O(k^{-3/2})$, just like the linear FEM for the Helmholtz 
problem (cf. \cite{Wu11}). The results of this subsection indicate 
that both methods satisfy the following pre-asymptotic error bounds (cf. Remark~\ref{r5.3}(a)):
\begin{align*}
\Bigl\{\norm{\bfE-\bfE_h}_{H(\curl,\cT_h)},\quad 
\norm{\bfE-\Eheem}_{H(\curl,\Om)} \Bigr\} \le C_1kh+C_2k^3h^2.
\end{align*} 
\begin{figure}[ht]
\centerline{
\includegraphics[scale=0.48]{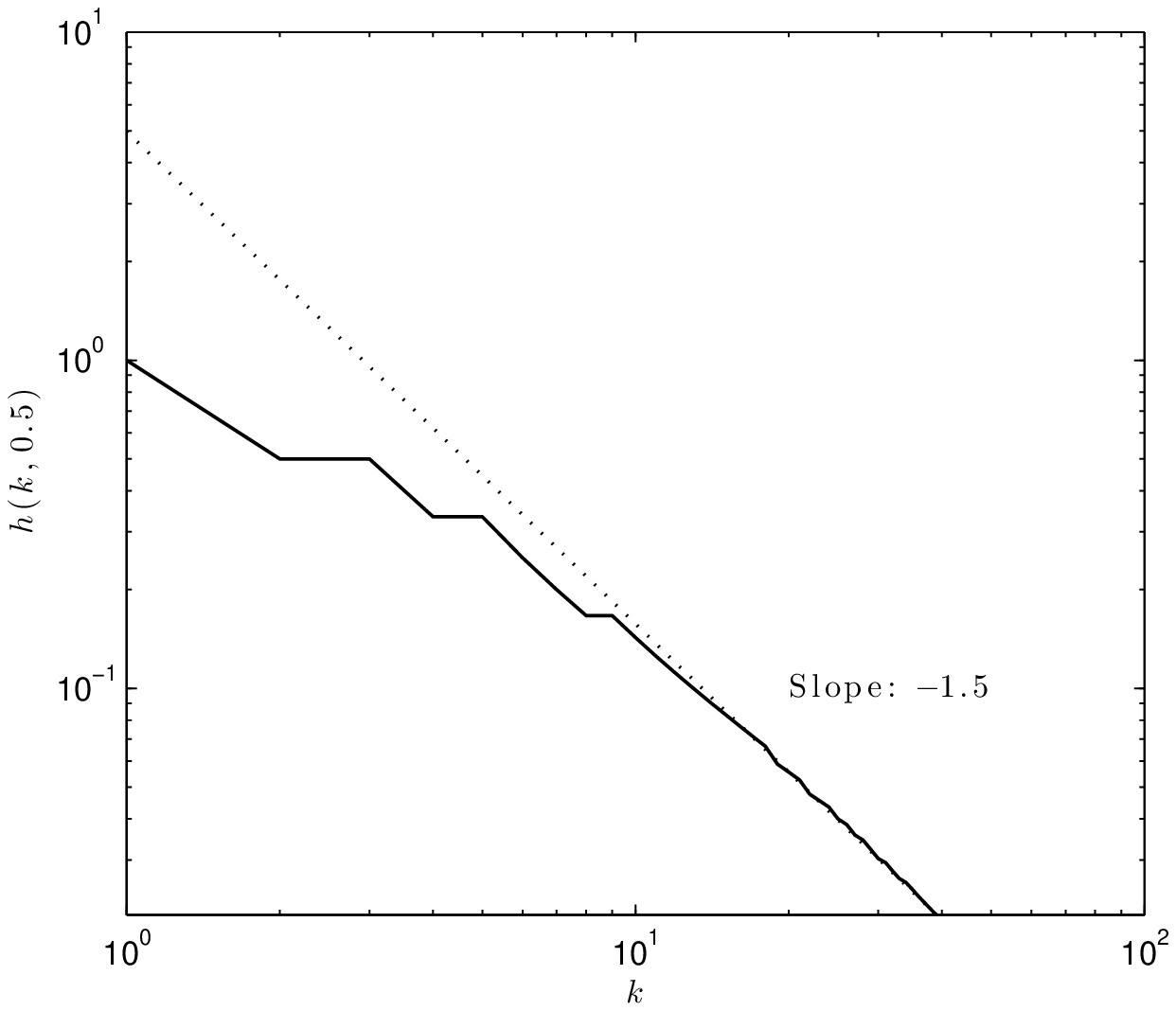}
\includegraphics[scale=0.48]{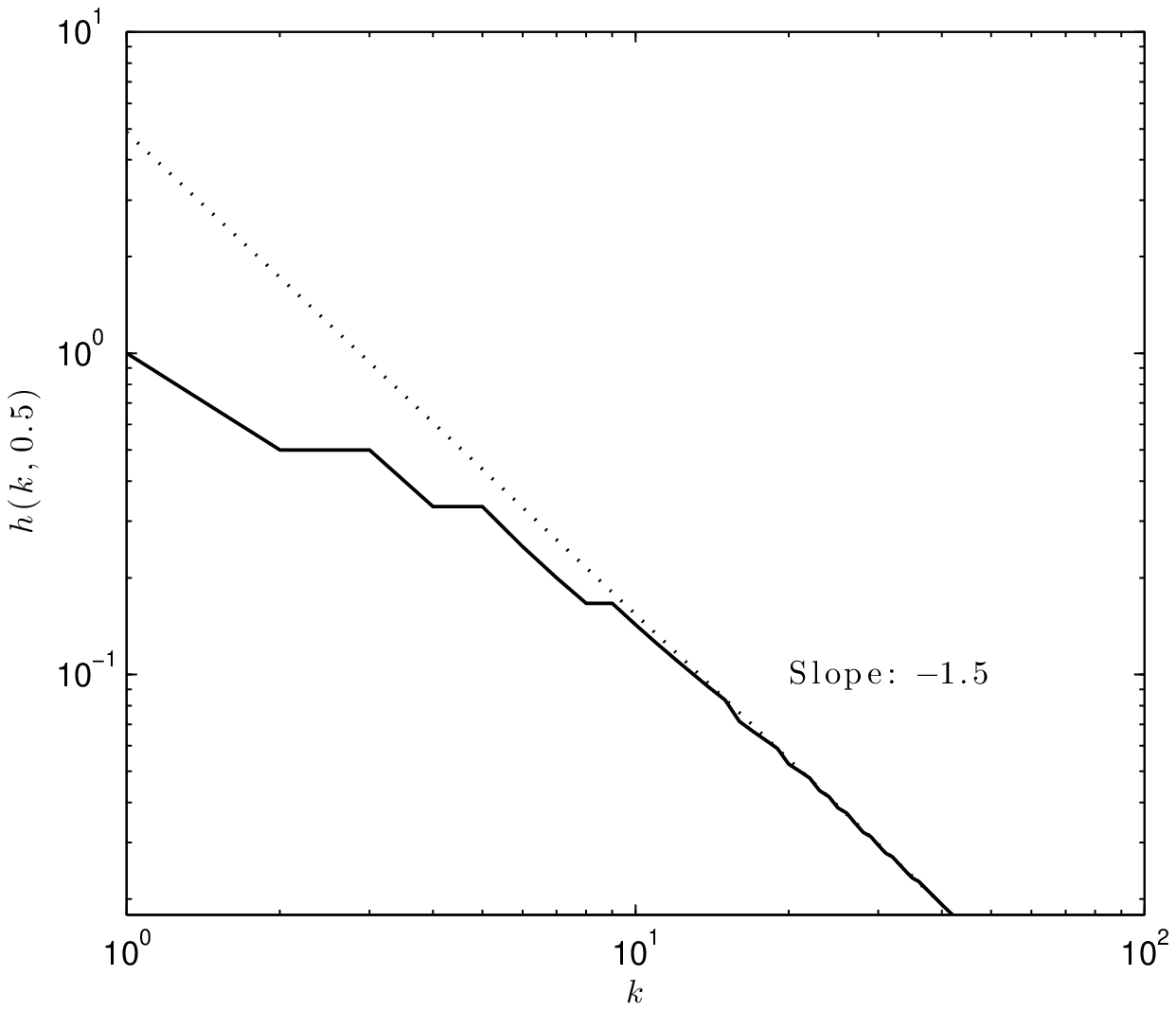}
}
\caption{$h(k,0.5)$ versus $k$ for the IPDG solution (left) with parameters 
given in \eqref{e7.4} and for the edge element solution (right), respectively. 
The dotted lines give lines of slope $-1.5$ in the log-log scale.}\label{ferr3}
\end{figure}

\subsection{Reduction of the pollution effect}\label{ssec-3}
In this subsection, we show that appropriate choice of the penalty parameters
can significantly reduce the pollution error of the IPDG method. 
We use the following parameters:
\begin{equation}\label{e7.8}
\ga_{0,\cF}\equiv\ga_0=100\quad\text{ and }\quad \i\ga_{1,\cF}
\equiv\i\ga_1=0.08+0.01\i \quad\forall \cF\in\cE_h^I.
\end{equation}
We remark that $\i\ga_{1,\cF}$ is simply chosen from the set
$\set{0.01(p+q\i), -50\le p,q\le 50}$ to minimize the relative error of the
IPDG solution in $H(\curl)$-norm with $\ga_0=100$ for
wave number $k=20$ and mesh size $h=1/10$. The optimal penalty parameter 
can also be obtained by the dispersion analysis (cf.  \cite{Ains04}) 
and will be considered in a future work.

The relative error of the IPDG solution with parameters given in \eqref{e7.8}
and the relative error of the edge element interpolant are displayed in the
left graph of Figure~\ref{ferr12b}. The IPDG method with  parameters given
in \eqref{e7.8} is much better than both the IPDG method using  parameters given
in \eqref{e7.4} and the EEM (cf. Figure~\ref{ferr1} and Figure~\ref{ferr2}). 
The relative error does not increase much with the change of $k$ along 
line $k h=1$ for $k\le 30$. But this does not mean that the pollution error
has been eliminated.
\begin{figure}[ht]
\centerline{
\includegraphics[scale=0.46]{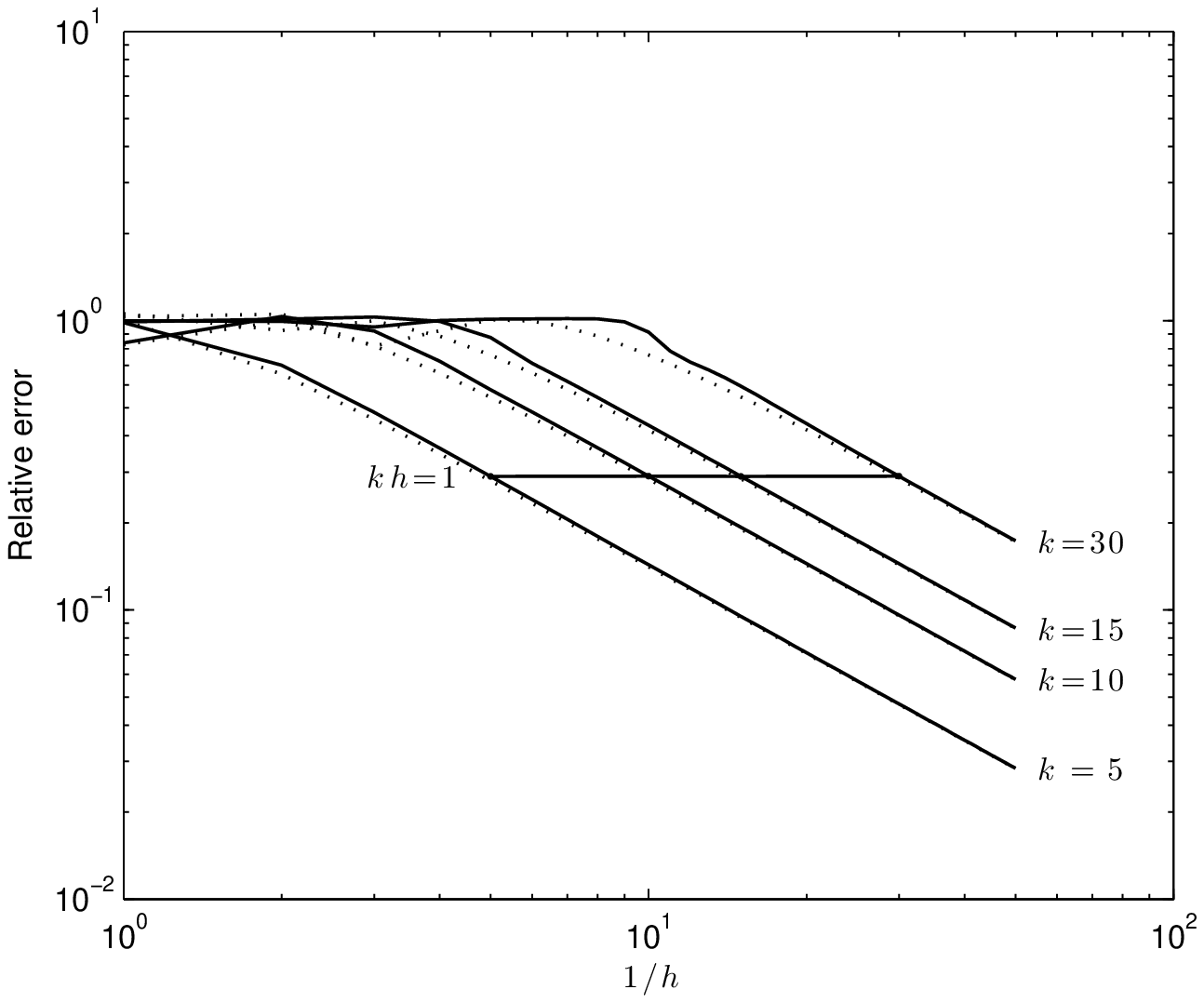}
\includegraphics[scale=0.46]{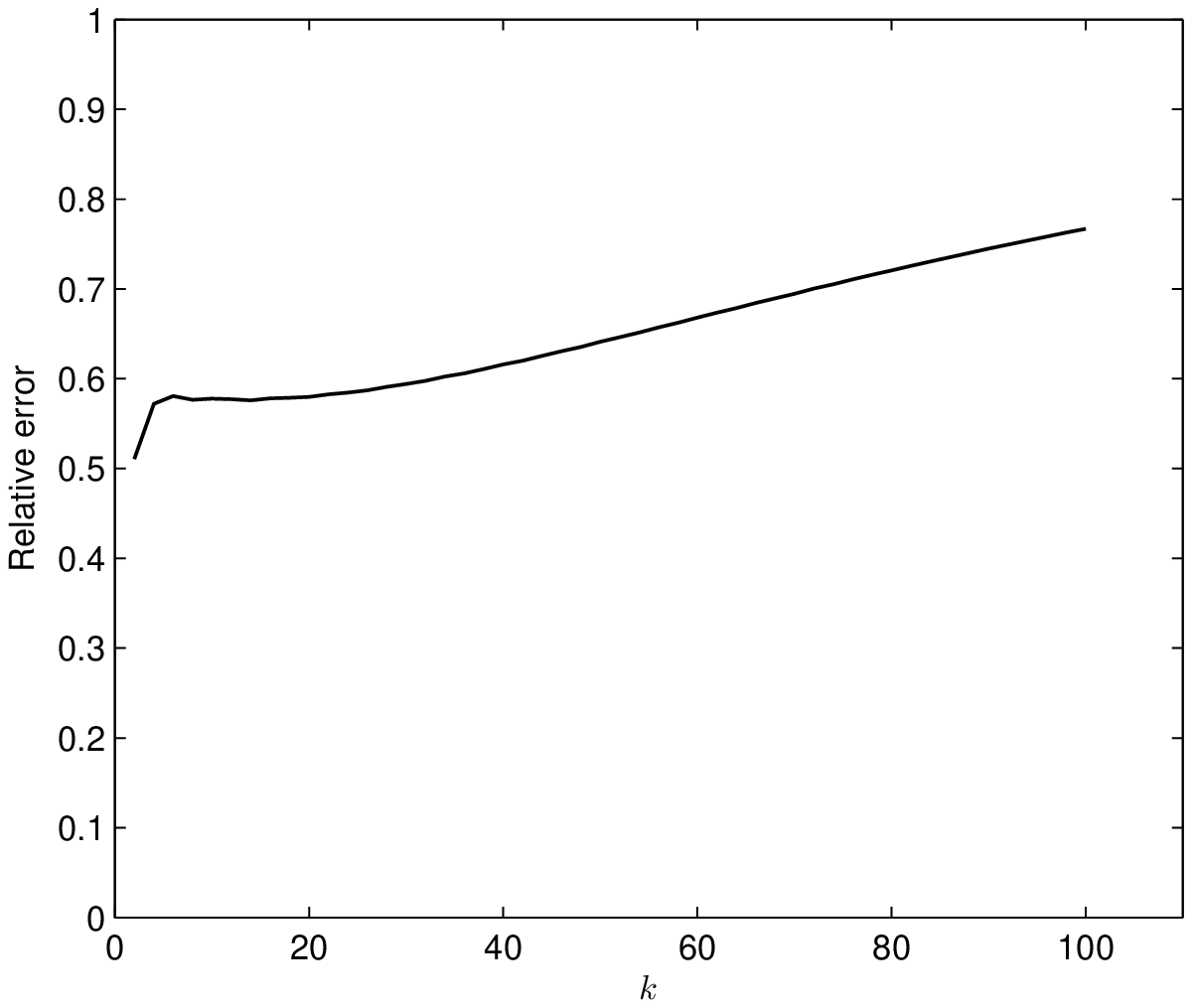}
}
\caption{Left graph: the relative error of the IPDG solution  with parameters given
in \eqref{e7.8} (solid) and the relative error of the edge element
interpolant (dotted) in $H(\curl)$-norm for $k=5, k=10, k=15,$ and $k=30$,
respectively. Right graph: the relative error of the IPDG solution with
parameters given in \eqref{e7.8} in $H(\curl)$-norm computed for $k=2, 4, 
\cdots, 72$ with mesh size $h$ determined by $kh=2$.}\label{ferr12b}
\end{figure}
For more detailed observation, the relative error of the IPDG solution with
parameters given in \eqref{e7.8} in $H(\curl)$-norm computed for $k=2, 4, \cdots, 72$
with mesh size $h$ determined by $kh=2$, are plotted in
the right graph of Figure~\ref{ferr12b}. It is shown that the pollution error is
reduced significantly.

Figure~\ref{ferr3b} plots $h(k,0.5)$, the critical mesh size with respect to the 
relative tolerance $50\%$, versus $k$ for the IPDG method with parameters given
in \eqref{e7.8}. We recall that $h(k,0.5)$ is the maximum mesh size such that 
the relative error of the IPDG solution in $H(\curl)$-norm is less than or 
equal to $50\%$. The decreasing rate of $h(k,0.5)$ in the log-log scale is 
less than $-1.5$, which means that the pollution effect is reduced.  
\begin{figure}[ht]
\centerline{
\includegraphics[scale=0.46]{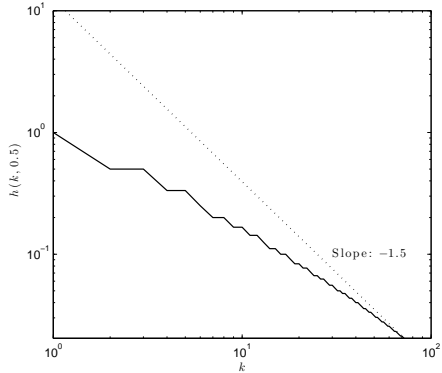}
}
\caption{$h(k,0.5)$ versus $k$ for the IPDG method with parameters given
in \eqref{e7.8}. The dotted line gives a line of slope $-1.5$ in the 
log-log scale.}\label{ferr3b}
\end{figure}

For more detailed comparison between the continuous interior penalty finite 
element method (CIP-FEM) and the FEM, we consider the problem
\eqref{e7.1}--\eqref{e7.2} with wave number $k=36$.
The real parts of $\bfE_{hx}(0.5,0.5,z)$ with parameters given in
\eqref{e7.8} (left, solid), $\bfE_{hx}^{\mathrm{EEM}}(0.5,0.5,z)$ (right, solid), 
and $\bfE_x(0.5,$ $0.5,z)$ (dotted) with mesh sizes $h=1/18$ 
and $1/36$ are plotted in Figure~\ref{fz}. Here $\bfE_{hx}$, 
$\bfE_{hx}^{\mathrm{EEM}}$, and $\bfE_x$ are the $x$ components of the 
IPDG solution,  the edge element solution, and the exact solution, respectively.
The shape of the IPDG solution is roughly same as that of the exact solution
for $h=1/18$ and matches very well for $h=1/36$. While the edge element 
solution has a wrong shape for $h=1/18$ and $z>0.5$ and has a correct
shape for $h=1/36$ but suffers an apparent phase error. 
\begin{figure}[ht]
\centerline{\includegraphics[scale=0.48]{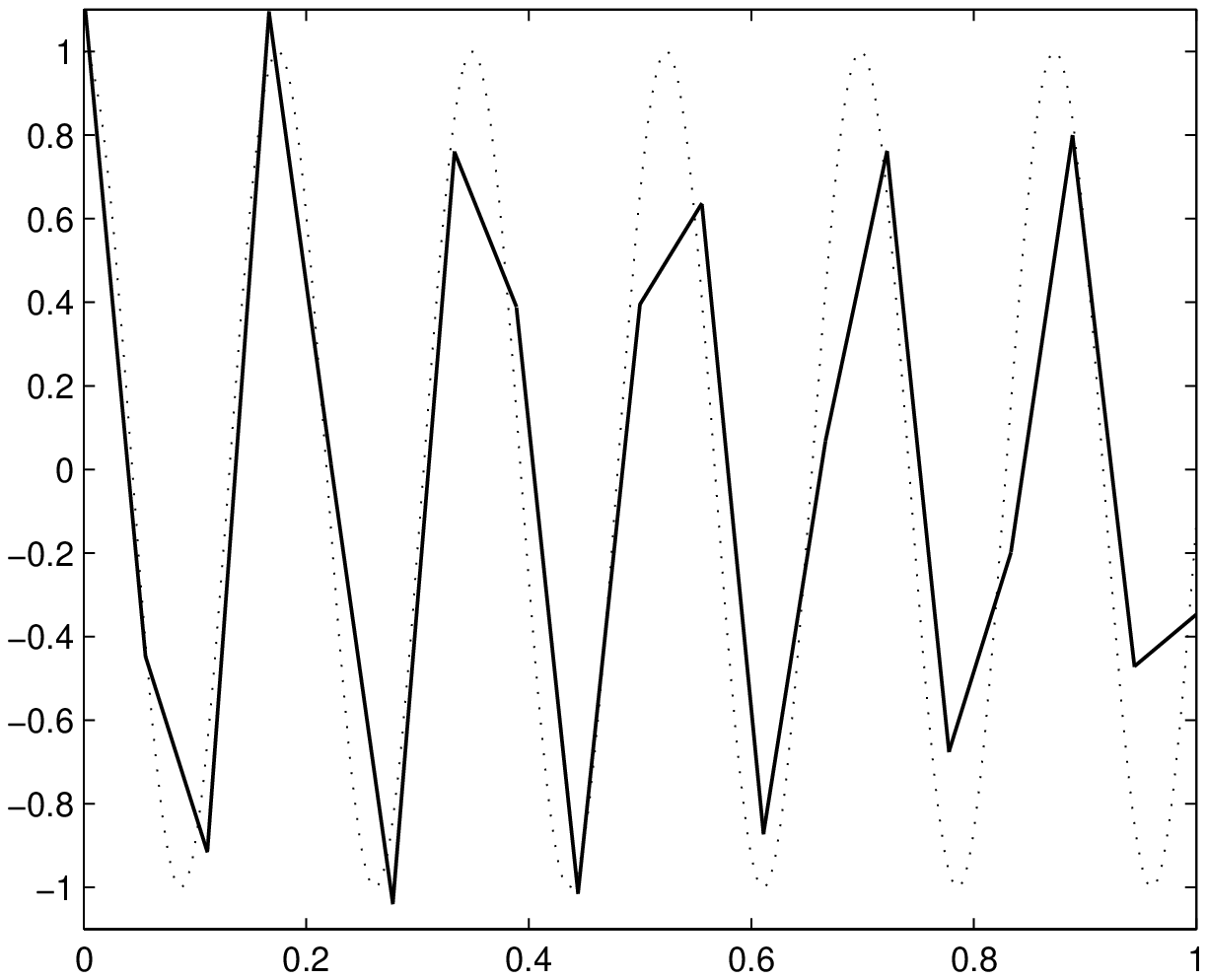}
\includegraphics[scale=0.46]{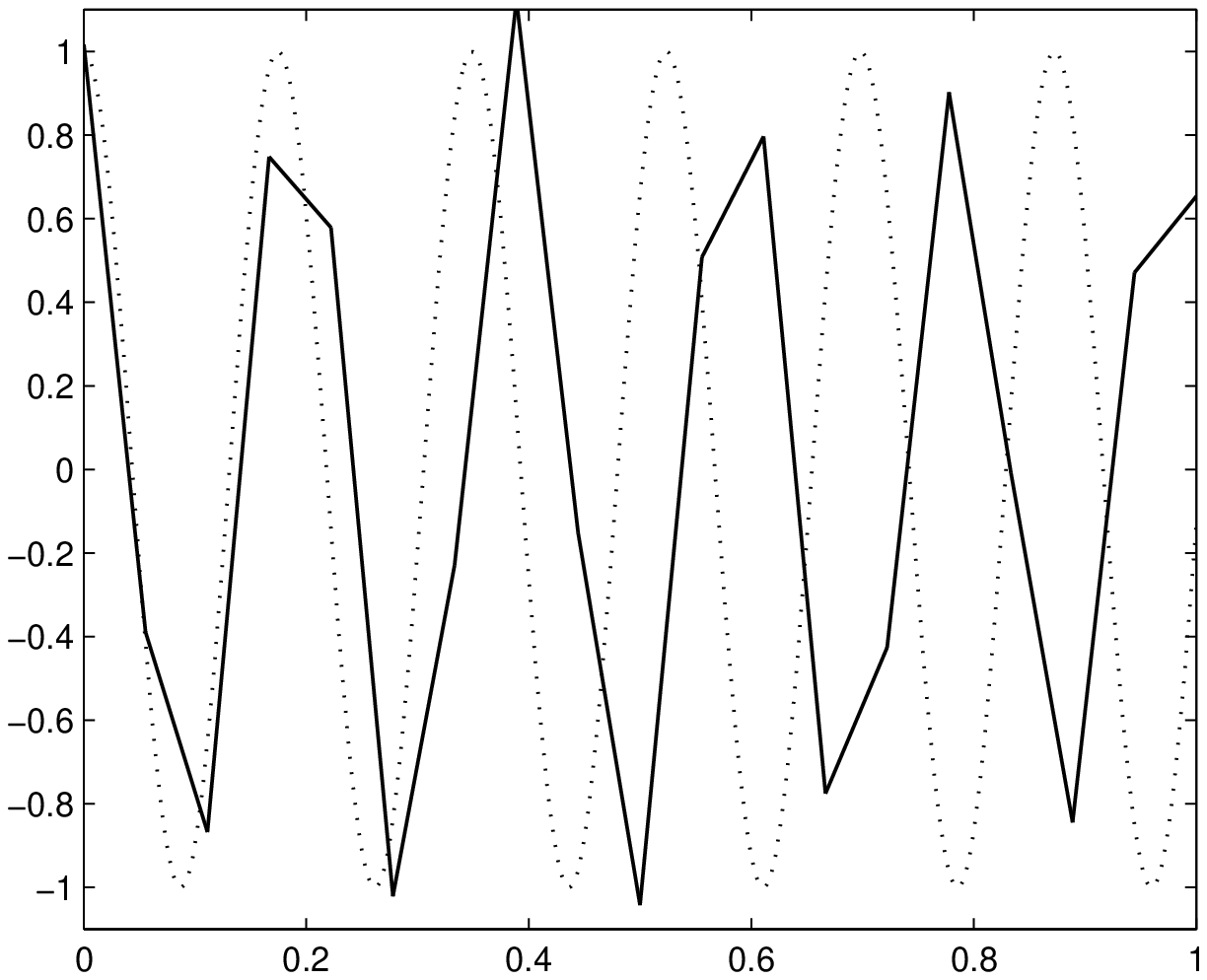}}
\centerline{\includegraphics[scale=0.48]{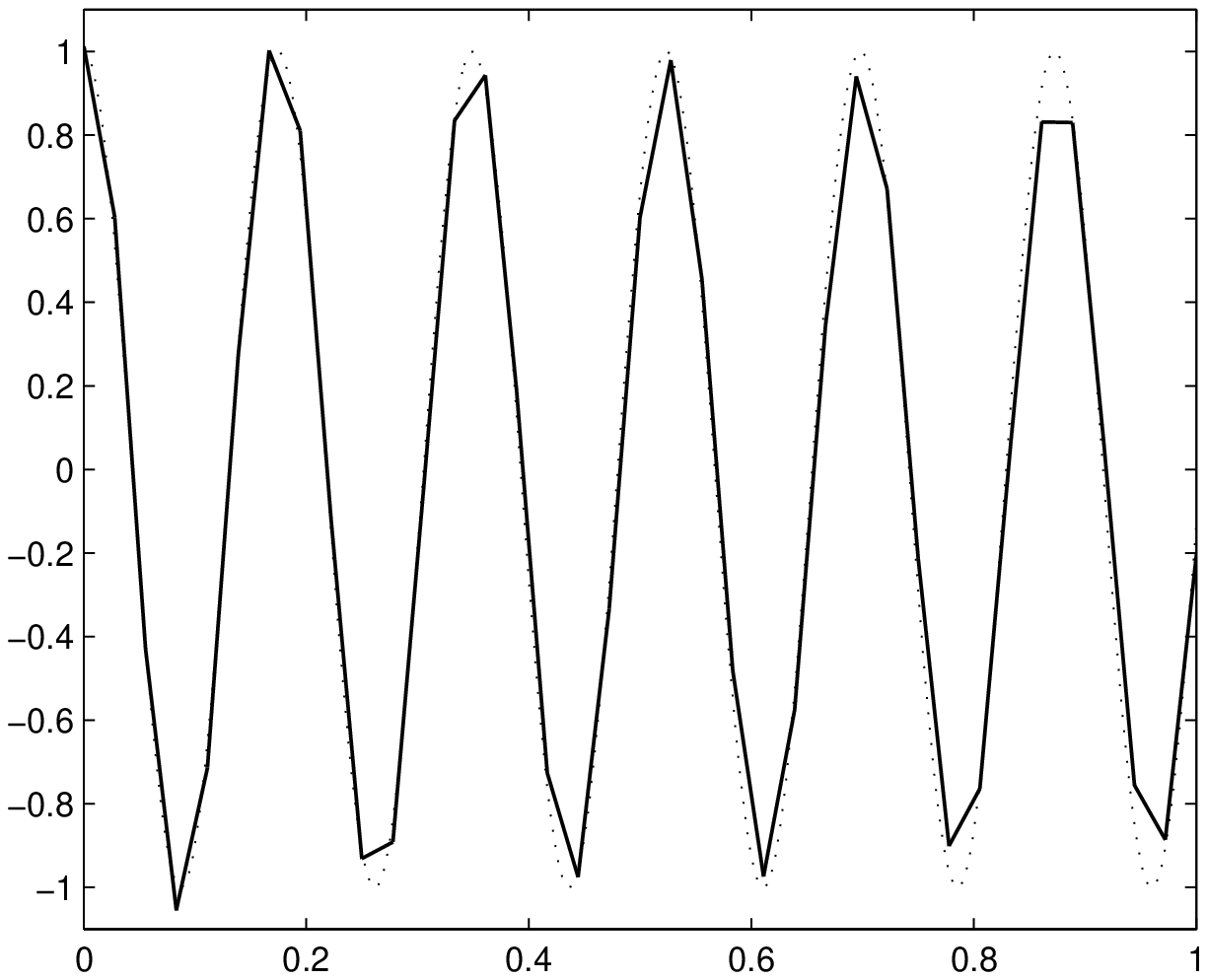}
\includegraphics[scale=0.46]{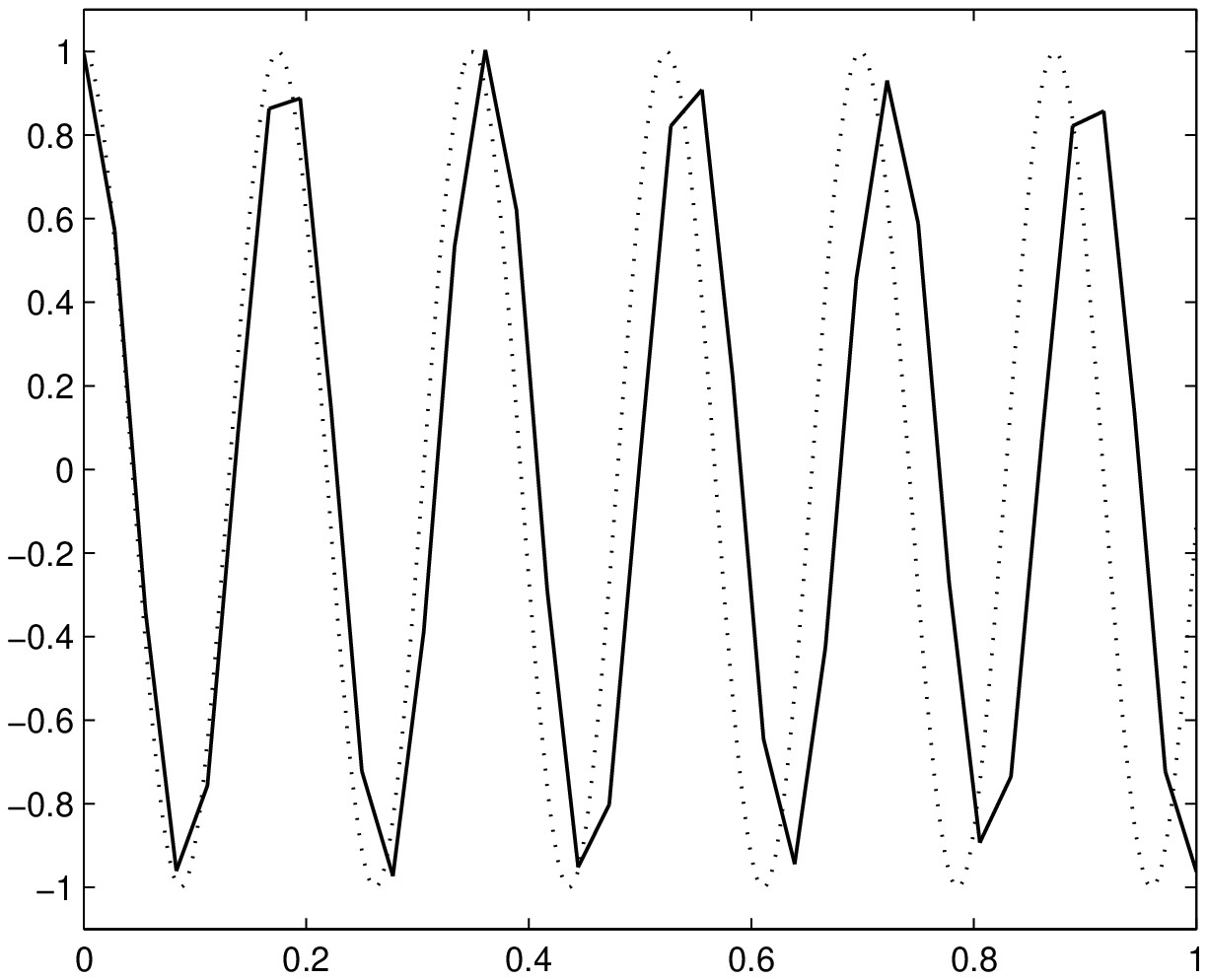}}
\caption{The real parts of $\bfE_{hx}(0.5,0.5,z)$  with parameters
given in \eqref{e7.8} (left, solid), $\bfE_{hx}^{\mathrm{EEM}}(0.5,0.5,z)$ 
(right, solid), and $\bfE_x(0.5,0.5,z)$ (dotted) for $k=36$ and $h=1/18$, $1/36$, 
respectively.} \label{fz}
\end{figure}

Table~\ref{table} shows the numbers of total DOFs needed for $50$\%
relative errors in $H(\curl)$-norm for the edge element interpolant, the 
IPDG solution  with parameters given in \eqref{e7.8}, and the edge element
solution, respectively.  The IPDG method needs less DOFs than the EEM does 
for $k\ge 10$ and much less for large wave number $k$. 
\begin{table}[hbt]
\centering
\begin{tabular}{|l|r|r|r|r|r|}
  \hline
    $k$ & 10 & 20 & 30 & 40 & 50 \\\hline
    Interpolation & 1,764& 12,168  & 33,048 & 79,488 & 141,288  \\\hline
    IPDG & 2,592 & 20,736 & 69,984 & 187,500 & 393,216 \\\hline    
    EEM & 2,688 & 45,600 & 249,900 & 876,408  & 2,398,488 \\\hline 
  \end{tabular}
\caption{Numbers of total DOFs needed for 50\% relative errors
in $H(\curl)$-norm for the edge element interpolant, the IPDG solution with 
parameters given in \eqref{e7.8}, and the edge element solution respectively.}
\label{table} 
\end{table}

\textbf{Acknowledgments.} The authors would like to thank Dr. Huangxin Chen 
of Xiamen University of China for his helpful suggestions on the construction 
and analysis of the $\bH(\curl,\Ome)$-elliptic projection in section \ref{sec-5.1}.


\appendix
\section{Proof of \eqref{e5.11f}}
The proof follows the same lines
as those given in \cite[pages 502--505]{HPSS05} and in \cite{Monk03,ZSWX09}. Let 
\[
U_h=\bigl\{ v_h\in H^1(\Om);\, v_h|_K\in P_2(K), \forall K\in\T \bigr\}
\] 
be the $H^1$-conforming linear finite element space. It follows from \eqref{e5.1} that $\bfE-\widetilde{\bfE}_h$ 
is discrete divergence-free, that is,
\[
\bigl(\bfE-\widetilde{\bfE}_h, \na \vp_h\bigr)_\Ome=0
\qquad\forall \vp_h\in U_h.  
\]
Notice that $\bfPhi_h^c\in \bfV_h\cap \bfH(\curl,\Ome)$, 
we have the following discrete Helmholtz decomposition of
$\bfPhi_h^c$:
\begin{align}
\bfPhi_h^c=\bfw_h+\na r_h,
\end{align}
where $r_h\in U_h$ and $\bfw_h\in \bfV_h\cap \bfH(\curl,\Ome)$ is also discrete 
divergence-free. It is easy to check that 
\begin{align*}
\norml{\na r_h}{\Om}\le\norml{\bfPhi_h^c}{\Om},
\qquad \norml{\bfw_h}{\Om}\ls \norml{\bfPhi_h^c}{\Om}.
\end{align*}
Then from  \cite[Lemma~7.6]{Monk03} and on noting that the domain $\Om$ is convex, there exists $\bfw\in \bfH^1(\Om)$ such that $\bfw\cdot\bn=\mathbf{0}$ on $\Ga$ and 
\begin{align}
&\curl\bfw=\curl\bfw_h,\quad \Div\bfw=0,\label{e6.8}\quad
 \norml{\bfw_h-\bfw}{\Om} 
\ls h\norml{\curl \bfPhi_h^c}{\Om}.
\end{align}
Thus, it follows from the identity 
\begin{align*}
\bigl(\bfE-\widetilde{\bfE}_h, \bfPhi_h^c \bigr)_\Ome
=\bigl(\bfE-\widetilde{\bfE}_h, \bfw_h \bigr)_\Ome
=\bigl(\bfE-\widetilde{\bfE}_h, \bfw_h-\bfw \bigr)_\Ome
+\bigl(\bfE-\widetilde{\bfE}_h, \bfw \bigr)_\Ome
\end{align*}
that
\begin{align}\label{e6.9}
-\frac{\bigl(\bfE-\widetilde{\bfE}_h, \bfPhi_h^c \bigr)_\Ome}{\norml{\bfE-\widetilde\bfE_h}{\Ome}}
&\le \norml{\bfw_h-\bfw}{\Om}+\norml{\bfw}{\Om}\\
&\ls h\norml{\curl \bfPhi_h^c}{\Om}+\norml{\bfw}{\Om}\nn.
\end{align}

The first term on the right-hand side of \eqref{e6.9} can be bounded as follows:
\begin{align}\label{e6.10}
&h\norml{\curl\bfPhi_h^c}{\Om}
=h\norml{\curl(\widehat{\bfE}_h-\bfE+\bfE-\widetilde\bfE_h
+\bfPhi_h- \bfPhi_h^c)}{\Om}\\
\le& h\|\widehat{\bfE}_h- \bfE\|_{\bH(\curl,\Om)}
+h\|\bfE-\widetilde\bfE_h\|_{DG}
+h\|\curl(\bfPhi_h-\bfPhi_h^c)\|_{\bL^2(\cT_h)}\ls h^2\cR(\bfE).\nn
\end{align}
where we have used \eqref{e5.6}, \eqref{e5.8c}, and \eqref{e5.11a} to derive the last inequality.

To estimate $\norml{\bfw}{\Om}$, we appeal to a duality argument to be
described next. Let $\bfz$ be the solution of the following auxiliary problem:
\begin{align}
\curl\curl\bfz+\bfz&=\bfw \quad\mbox{in }\Omega, \qquad\curl\bfz\times\bfnu=\mathbf{0} 
\quad\mbox{on }\Gamma:=\partial\Ome. \label{e:z}
\end{align}
Noting that $\Om$ is convex, the above problem attains a unique solution $\bfz\in \bH^1(\curl,\Om)$ and satisfies the following regularity estimate  (cf. \cite{HMP10,Monk03})
\begin{equation}\label{C_la}
\|\bfz\|_{\bH^1(\curl,\Ome)} \ls \|\bfw\|_{\bL^2(\Ome)}.
\end{equation} 

Define sesquilinear forms
\begin{align*}
A(\bfu,\bfv)&:=(\curl\bfu,\curl\bfv)_\Ome +(\bfu,\bfv)_\Ome \qquad\forall \bfu,\bfv\in \hat{\bcV},\\
A_h(\bfu,\bfv)&:=b_h(\bfu,\bfv)+(\bfu,\bfv)_\Ome \qquad\forall \bfu,\bfv\in \bfV.
\end{align*}
Let $\bfz_h^c\in \bfV_h\cap \bfH(\curl,\Om)$ and $\bfz_h\in \bfV_h$ denote the edge finite element approximation 
and the IPDG approximation to $\bfz$, respectively, that is,
\begin{align*}
A(\bfv_h,\bfz_h^c)&=(\bfv_h,\bfw)_\Ome 
\qquad\forall\bfv_h\in \bfV_h\cap \bfH(\curl,\Om),\\
A_h(\bfv_h,\bfz_h)&=(\bfv_h,\bfw)_\Ome \qquad\forall\bfv_h\in \bfV_h.
\end{align*}
It can be shown that there hold the following estimates (cf. \eqref{e5.6}):
\begin{align}\label{C1}
\|\bfz-\bfz_h^c\|_{\bH(\curl,\Ome)}  
&\ls h\, \|\bfz\|_{\bH^1(\curl,\Ome)} 
\ls h\, \norml{\bfw}{\Om}, \nn  \\
\norme{\bfz-\bfz_h^c},\quad\norme{\bfz-\bfz_h} &\ls h\,(1+\ga_1)^\frac12 \|\bfz\|_{\bH^1(\curl,\Ome)}
\ls h\,(1+\ga_1)^\frac12  \norml{\bfw}{\Om}. \nn
\end{align}

Since
\begin{align*}
\norml{\bfw}{\Om}^2&=A(\bfw,\bfz)=A(\bfw,\bfz-\bfz_h^c)+A(\bfw,\bfz_h^c), 
\end{align*} 
on noting that $\bfw, \bfw_h\in \bfH(\curl,\Ome)$, from \eqref{e6.8} and \eqref{e6.10}, we have
\begin{align*}
A(\bfw,\bfz-\bfz_h^c)&=A(\bfw-\bfw_h,\bfz-\bfz_h^c)
= (\bfw-\bfw_h,\bfz-\bfz_h^c)_\Ome\\
&\ls h\norml{\curl\bfPhi_h^c}{\Om}\, h\,\|\bfw\|_{\bL^2(\Om)}\ls h^3\cR(\bfE) \norml{\bfw}{\Om}.
\end{align*}
 On the other hand,
\begin{align*}
A(\bfw,\bfz_h^c)&=(\curl \bfw,\curl\bfz_h^c)_\Ome
+(\bfw,\bfz_h^c)_\Ome =\bigl(\curl \bfPhi_h^c,\curl\bfz_h^c\bigr)_\Ome
+(\bfw,\bfz_h^c)_\Ome 
 \db\\
&=A(\bfPhi_h^c,\bfz_h^c)
+\bigl(\bfw-(\bfw_h+\na r_h),\bfz_h^c\bigr)_\Ome=A(\bfPhi_h^c,\bfz_h^c)+(\bfw-\bfw_h,\bfz_h^c)_\Ome.
\end{align*}
From the definitions of $A, A_h$ and $b_h$, we get
\begin{align*}
&A(\bfPhi_h^c,\bfz_h^c)
=A_h(\bfPhi_h^c,\bfz_h^c)
+\i J_1(\bfPhi_h^c,\bfz_h^c)\\
&\quad
=A_h(\bfE-\widehat{\bfE}_h,\bfz_h^c)
+A_h(\widetilde\bfE_h-\bfE,\bfz_h^c)
+A_h(\bfPhi_h^c-\bfPhi_h,\bfz_h^c)
+\i J_1(\bfPhi_h^c,\bfz_h^c)\\
&\quad
=A_h(\bfE-\widehat{\bfE}_h,\bfz_h^c)
+A_h(\bfPhi_h^c-\bfPhi_h,\bfz_h^c)
+\i J_1(\bfPhi_h^c,\bfz_h^c-\bfz).
\end{align*}
Therefore
\begin{align*}
A(\bfw,\bfz_h^c)=&A_h(\bfE-\widehat{\bfE}_h,\bfz_h^c-\bfz)
+A_h(\bfE-\widehat{\bfE}_h,\bfz)+A_h(\bfPhi_h^c-\bfPhi_h,\bfz_h^c-\bfz+\bfz-\bfz_h)
\\
&+A_h(\bfPhi_h^c-\bfPhi_h,\bfz_h)+\i J_1(\bfPhi_h^c,\bfz_h^c-\bfz)
+(\bfw-\bfw_h,\bfz_h^c)_\Ome.
\end{align*}
Since
\begin{align*}
A_h(\bfE-\widehat{\bfE}_h,\bfz)=(\bfE-\widehat{\bfE}_h,\bfw)_\Ome,
\qquad 
A_h(\bfPhi_h^c-\bfPhi_h,\bfz_h)
=(\bfPhi_h^c-\bfPhi_h,\bfw)_\Ome,
\end{align*}
we have from Lemma~\ref{lem5.1} and the local trace inequality,
\begin{align*}
&A(\bfw,\bfz_h^c) \ls\norme{\bfE-\widehat{\bfE}_h} \norme{\bfz_h^c-\bfz}
+\norml{\bfE-\widehat{\bfE}_h}{\Om}\norml{\bfw}{\Om}\\
&\hskip 46pt+ \norme{\bfPhi_h^c-\bfPhi_h}
\big(\norme{\bfz_h^c-\bfz}+\norme{\bfz-\bfz_h}\big)
+\norml{\bfPhi_h^c-\bfPhi_h}{\Om} \norml{\bfw}{\Om}\\
&\hskip 46pt+\ga_1\norml{\curl\bfPhi_h^c}{\Om}
\big(\norml{\curl(\bfz_h^c-\bfz)}{\Om}+h\|\curl(\bfz_h^c-\bfz)\|_{H^1(\T)}\big)
\\
&\hskip 46pt+\norml{\bfw-\bfw_h}{\Om} \norml{\bfz_h^c}{\Om}\\
&\ls\norml{\bfw}{\Om}\Big(h\,(1+\ga_1)^\frac12  \norme{\bfE-\widehat{\bfE}_h}
+\norml{\bfE-\widehat{\bfE}_h}{\Om}+h\,(1+\ga_1)^\frac12  
\norme{\bfPhi_h^c-\bfPhi_h} \\
&\quad+\norml{\bfPhi_h^c-\bfPhi_h}{\Om}+\ga_1h\norml{\curl\bfPhi_h^c}{\Om}+h\norml{\curl \bfPhi_h^c}{\Om}\Big).
\end{align*}
Moreover, from \eqref{e5.11a}, \eqref{e5.11}, $\ga_0\gtrsim 1$,  and the local trace inequality, we get
\begin{align*}
\norme{
\bfPhi_h^c-\bfPhi_h}&\ls (1+\ga_1)^\frac12
\|\curl(\bfPhi_h-\bfPhi_h^c)\|_{\bL^2(\cT_h)}+\|\bfPhi_h\|_{DG}+\norml{\bfPhi_h^c-\bfPhi_h}{\Om}\\
&\ls (1+\ga_1)^\frac12 h\,\cR(\bfE).
\end{align*}
Thus, it follows from \eqref{e5.11}, \eqref{e6.10}, \eqref{e5.11a}, and the above estimate that
\begin{align*}
A(\bfw,\bfz_h^c) \ls &(1+\ga_1) h^2 \cR(\bfE)\norml{\bfw}{\Om}.
\end{align*}
Then we obtain the following estimates for $\norml{\bfw}{\Om}$:
\begin{align*}
\norml{\bfw}{\Om} \ls (1+\ga_1)h^2\cR(\bfE),
\end{align*}
which together with \eqref{e6.9} and \eqref{e6.10} implies that \eqref{e5.11f} holds.
The proof is complete.

\end{document}